\documentclass[12pt]{article}

\usepackage{latexsym}
\usepackage{amsopn,amscd}
\usepackage{amsfonts}
\usepackage{amssymb}
\usepackage{amsthm}
\usepackage{amsmath}
\usepackage{a4wide}
\usepackage[square,authoryear]{natbib}

\newtheorem{Theorem}{Theorem}[section]
\newtheorem{Proposition}[Theorem]{Proposition}

\newtheorem{Corollary}[Theorem]{Corollary}
\newtheorem{Remark}[Theorem]{Remark}
\newtheorem{Example}[Theorem]{Example}

\def\QQ{P}
\def\empha{\em}

\def\bemphas{}
\def\pemphas{}
\def\fra{\mathfrak}
\def\Sigm{\mathcal S}
\def\MMM{M}
\def\Cat{\mathcal D}
\def\chain{\mathcal C}
\def\chainB{\mathcal C_{\BBB}}

\def\VV{\mathbf V}
\def\sou{s}
\def\tar{t}
\def\Cat{\mathcal D}
\def\rhoo{\rho}
\def\aalpha{\alpha}
\def\ccc{c}
\def\uuu{u}
\def\GGG{\Omega}
\def\GGGG{\Omega}

\def\BBB{B_{\GGG}}
\def\BB{B_{\GGG}}
\numberwithin{equation}{section}

\begin{document}

\title{\bf Equivariant cohomology over Lie groupoids and  Lie-Rinehart algebras}

\author{
J.~Huebschmann
\\[0.3cm]
 USTL, UFR de Math\'ematiques\\
CNRS-UMR 8524
\\
59655 Villeneuve d'Ascq C\'edex, France\\
Johannes.Huebschmann@math.univ-lille1.fr
 }
\maketitle
\begin{abstract}
{Using the language and terminology of relative homological algebra, 
in particular that of derived functors,
we introduce equivariant cohomology over a general Lie-Rinehart algebra
and equivariant de Rham cohomology over a locally trivial
Lie groupoid in terms of suitably defined monads 
(also known as triples) and  the associated standard constructions.
This extends a characterization of equivariant de Rham
cohomology in terms of derived functors developed earlier
for the special case where the Lie groupoid is an ordinary
Lie group, viewed as a Lie groupoid with a single object;
in that theory over a Lie group, the ordinary Bott-Dupont-Shulman-Stasheff
complex arises as an a posteriori object.
We prove that, given a locally trivial Lie groupoid $\GGG$
and a smooth $\GGG$-manifold $f\colon M \to \BBB$ 
over the space $\BBB$ of objects of $\GGG$,
the resulting $\GGG$-equivariant de Rham theory 
of $f$ boils down to the
ordinary equivariant de Rham theory of a
vertex manifold $f^{-1}(q)$ relative to the vertex group  
$\GGG^q_q$, for any vertex $q$
in the space $\BBB$ of objects of $\GGG$;
this implies that the 
equivariant de Rham cohomology introduced here 
coincides with the 
stack de Rham cohomology of the associated
transformation groupoid whence this stack de Rham cohomology
can be characterized as a relative derived functor.
We introduce a notion of cone on a Lie-Rinehart algebra
and in particular 
that of cone on a Lie algebroid. This cone is an indispensable tool
for the description of the requisite monads.}
\end{abstract}

\noindent
\\[0.3cm]
\\[0.3cm]

{\tableofcontents}

\section{Introduction}

Equivariant de Rham cohomology over a Lie group 
is usually defined in terms of a
certain bar-de Rham bicomplex; this bicomplex can be traced back, at least, to
\cite{bottone}, \cite{botshust}, \cite{duponone}, \cite{gsegatwo}, 
\cite{shulmone} 
and  serves, in de Rham theory, as a replacement for the Borel 
construction. In modern terminology, this bar-de Rham 
bicomplex yields the de Rham cohomology of the associated differentiable stack.
However the definition in terms of the bar-de Rham bicomplex hides
the origins of equivariant cohomology, which is the theory of derived
functors. 
Indeed the search
for understanding the impact of  symmetries encapsulated
in a group on global invariants
such as cohomology led to the 
cohomology of groups and thereafter to equivariant cohomology in general.
Once the masters had isolated the cohomology of groups,
they realized that the theory can be formalized in the language of 
derived functors. This led eventually to the description of derived functors
in terms of (co-)monads, also known as (co)triples, but we will
exclusively use the monad-comonad terminology.

While the cohomology of groups is nowadays well understood as an instance
of the theory of derived functors, 
it is less common to view equivariant cohomology 
as an instance of the theory of derived functors.
However, given a topological group $G$
and a $G$-space $X$, over an arbitrary commutative ground ring $R$,
the $G$-equivariant singular cohomology of $X$ can be characterized
as the differential $\mathrm{Tor}^{C_*(G)}(R,C^*(X))$.
Here $C_*(G)$ refers to the differential graded $R$-algebra of singular chains
on $G$ and $C^*(X)$ to the $R$-algebra of ordinary singular cochains on $X$,
turned into a $C_*(G)$-module through the $G$-action on $X$.
The standard object calculating this differential Tor is 
a bar construction; indeed, the  Borel construction is well known to be the
appropriate geometric analog of the corresponding bar construction.
There is a perfectly reasonable (co)monad which leads
to that differential Tor, and
the appropriate comonad yields the homotopy quotient or Borel construction
as an \lq\lq a posteriori\rq\rq\ object calculating equivariant cohomology.
It is, perhaps, also worthwhile noting that, within the
framework of relative homological algebra,
the description of the $G$-equivariant cohomology of $X$
as the differential $\mathrm{Tor}^{C_*(G)}(R,C^*(X))$
is self-explanatory and can be understood even without
explicit reference to the theory of monads etc.
In \cite{koszul} we have shown that homological perturbation theory,
applied to that relative derived functor, yields a conceptual
explanation of the phenomenon of Koszul duality.

In \cite{koszultw} we have introduced a monad 
which yields the 
equivariant de Rham cohomology with respect to a Lie group
as a suitably defined relative derived Ext-functor,
and we have also developed the corresponding
infinitesimal theory.
To capture 
global properties
which rely on symmetries of a certain geometric situation
more general than symmetries that can be encoded in a group,
e.~g. symmetries of a family of spaces rather than just of a space,
in the present paper we start building similar equivariant theories 
for groupoids and Lie-Rinehart algebras in the framework of
relative homological algebra.

The notions of injective module and injective resolution are fundamental.
They have been developed by the masters to cope with situations where
projective modules are not available.
In contrast to the case for projective
resolutions, the mechanism of {\em injective\/} resolutions 
is available in various situations where there are no projective objects;
also  the mechanism of injective resolutions can be  adapted
to take account of additional structure, e.~g. that of
{\em differentiability\/} of a module over a Lie group $G$;
indeed, given a Lie group $G$, the appropriate way to resolve an object of the
category of differentiable $G$-modules 
is by means of a {\em differentiably injective resolution\/}
in the sense of \cite{hochmost}.

The naive attempt to generalize 
equivariant cohomology over a Lie group
to a theory over a general Lie groupoid is bound to fail: There 
is no obvious way
to extend the bar-de Rham bicomplex
to a complex over a Lie groupoid to arrive at the desired cohomology theory.
Complications arise since the standard approach to
a relative derived functor in terms of 
the nonhomogeneous resolution associated with
the appropriate adjunction breaks down for groupoids.
This difficulty resides in the fact that, as observed
by K. Mackenzie in a remark on p.~285 of \cite{mackeigh},
for a general topological groupoid,
there is no analog of the action used 
by Hochschild and Mostow on p.~369 of \cite{hochmost}
which, for a Lie group $G$ and a vector space $V$, 
on the vector space $C^{\infty}(G,V)$ of smooth $V$-valued functions on $G$,
is given by left (or right) translation.
The putative extension of this kind of action would not be compatible
with the variance constraint imposed by the groupoid action.
See also Remark \ref{cd5} below.
Hence, for a general groupoid,
the ordinary naive construction of a relatively injective module
breaks down. 
The fact that,
over a general locally trivial locally compact groupoid, any continuous
module can be continuously embedded into a continuously injective one
is thus not immediate; it
has been established by K. Mackenzie in Theorem 1 of \cite{mackeigh}.

We will introduce equivariant cohomology over a general Lie-Rinehart algebra
and equivariant de Rham theory over a locally trivial
Lie groupoid in terms of suitably defined monads 
and the associated dual standard  constructions.
As for the terminology, we recall that, for purely formal 
or historical reasons,
one refers to the {\em standard construction associated with a comonad\/}
and to the 
{\em dual standard construction associated with a monad\/}.
These monads extend the constructions in the predecessor \cite{koszultw}
of the present paper
but, for reasons explained above, 
we cannot simply extend the monad
used in \cite{koszultw} to characterize ordinary
equivariant de Rham cohomology over a Lie group
and, indeed, in the present paper, we shall 
modify that monad so that the modified monad
extends to Lie groupoids.
We shall explain the details in Section 
\ref{equivgroupoid} below. In Theorem \ref{U6} 
we shall then prove that, 
given the locally trivial Lie groupoid $\GGG$ over $\BBB$ 
and the left $\GGG$-manifold
$f \colon M \to \BBB$ (fiber bundle endowed with a left $\GGG$-structure),
for any object $q\in \BBB$, our
$\GGG$-equivariant de Rham theory of $f$ 
is naturally isomorphic to the corresponding
ordinary $\GGG^q_q$-equivariant de Rham cohomology 
of the corresponding fiber $f^{-1}(q)$ over $q$; here $\GGG^q_q$
refers to the vertex group at $q$.
This is perfectly consistent with
Theorem 3 of \cite{mackeigh} which says that, for
a locally trivial topological groupoid, the rigid cohomology boils down to
the Hochschild-Mostow theory for any vertex group with the corresponding 
coefficients; see also Proposition \ref{prop9} below.

An indispensable tool in \cite{koszultw}
is the cone $C\fra g$ on an ordinary Lie algebra $\fra g$,
the cone being taken in the category of differential graded Lie algebras.
In the present paper, given the Lie groupoid $\GGG$, we 
accordingly introduce the {\em cone\/} on the Lie algebroid $\lambda_{\GGG}$
or, equivalently, the corresponding
cone on the corresponding Lie-Rinehart algebra
relative to the Atiyah sequence of the Lie groupoid;
see Section \ref{algebroids} below. This extension 
is not entirely immediate since the naive construction
of the cone on a Lie-Rinehart algebra,
which we explore in Section \ref{lierine} below,
is not the correct notion of cone needed for equivariant
de Rham cohomology relative to a Lie groupoid.
The cone defined in terms of the Atiyah sequence is an indispensable tool
for the description of the requisite monad defining 
the equivariant
de Rham cohomology relative to a Lie groupoid.

Monads, comonads, standard constructions and dual standard constructions
are explained in great detail in \cite{mackbtwo}.
The idea of a monad goes back to \cite{maclafiv}
and that of a standard construction to \cite{godebook}
where the 
canonical flasque resolution of a sheaf
is given as a suitable dual standard construction.
Despite its flexibility and vast range of possible
applications,  the theory of (co)monads, well known
in category theory circles, has hardly been 
used elsewhere in mathematics except in algebraic topology
where triples are quite common.
Our approach to equivariant de Rham theory 
relative to a general locally trivial Lie groupoid 
clearly shows that monads and
comonads deserve to be better known.

Here is a brief overview of the paper:
In Section \ref{monads} we recall some 
of the basic notions
of relative homological algebra used later in the paper, including
monads and the associated dual standard constructions, and we
illustrate these notions by means of certain examples involving groups and
Lie groups; there examples are of fundamental importance later in the paper.
In Section \ref{liegroupoids} we phrase the differentiable cohomology
of locally trivial Lie groupoids in a language tailored
to our purposes.
In Section \ref{comonads} we recall comonads and standard constructions,
including the familiar comonadic description as a simplicial $G$-set
or $G$-space of the universal object $EG$ 
associated with a Lie group $G$.
Given the $G$-module $V$, we will then recall that the 
cosimplicial object $\mathrm{Map}(EG,V)$ coincides
with the dual standard construction 
associated with the corresponding monad, cf. Proposition \ref{com1}.
In Section \ref{lierine} we develop an equivariant theory relative 
to a Lie-Rinehart algebra.
In Section \ref{ext} we explore extensions of Lie-Rinehart
algebras and introduce a notion of cone relative to an extension of
Lie-Rinehart algebras.
In Section \ref{algebroids} we use the material developed
before to introduce a notion of cone over the Lie algebroid associated
with a Lie groupoid.
In Section \ref{equivgroupoid} we then introduce equivariant
cohomology with respect to a locally trivial Lie groupoid
as a relative derived functor by means of a suitable monad.
In Section \ref{compare} we relate our approach with
other notions of Lie groupoid cohomology in the literature.
The term `cohomology' may have several meanings;
thus, rather than trying to develop the appropriate injective resolution,
one can  extend the construction of the
standard complex defining smooth Lie group cohomology
to the groupoid case;
the resulting
notion of cohomology has been explored in the literature,
e.~g. in \cite{crainic} and \cite{weinxu}. 
This kind of Lie groupoid cohomology 
does not involve differential forms at all; it 
arises from application of the functor $C^{\infty}$ to the nerve
of the groupoid,
a construction substantially different from the one we use.
While for smooth manifolds ordinary real cohomology
coincides with de Rham cohomology, under the present circumstances,
the relationship is less clear.
The naive attempt to recover the more direct definition
which consists in applying the functor $C^{\infty}$ to the nerve
is bound to fail since one would somehow 
try to compare apples and oranges.
Indeed, one cannot directly apply the de Rham functor
to the homotopy quotient arising from the nerve, 
and the de Rham cohomology of 
the associated stack is defined to be the total cohomology
of the resulting cosimplicial de Rham complex.
When the underlying Lie groupoid is actually an ordinary Lie group,
viewed as a Lie groupoid with a single object, this cosimplicial
de Rham complex comes down to the construction developed and explored
by Bott, Dupont, Shulman and Stasheff \cite{bottone}, \cite{botshust},
\cite{duponone}, \cite{shulmone}.
Prompted by a referee's report, we comment on the situation
in Section \ref{compare}.
In particular, exploiting the appropriate Morita equivalence,
we explain how our notion of equivariant cohomology
over a locally trivial Lie groupoid includes a
description of the corresponding
stack de Rham cohomology of the associated transformation groupoid
as a relative derived functor.
Morita equivalence
for the more direct definition is of course well known to hold, cf. e.~g.
\cite{moermcru}.
Under the circumstances of the present paper,
phrasing Morita equivalence is somewhat more delicate
since it would have to involve the coefficients as well,
and so far we do not understand Morita equivalence completely
for the constructions  developed here.
The equivalence between 
our notion of equivariant cohomology
and the corresponding
stack de Rham cohomology of the associated transformation groupoid
includes a version of Morita equivalence in a somewhat roundabout manner.
Under the circumstances of the paper,
the final form of Morita equivalence will presumably extend an old observation
of Seda's \cite{sedaone} to the effect that,
over a locally trivial topological groupoid,
the functor given by restriction to a vertex group has the functor which assigns to
a representation the associated induced fiber bundle
as its left-adjoint;
cf. also the proof of Proposition \ref{prop9} below. 
In a final section we discuss various open questions.

The appropriate categorical setting for many constructions
in the present paper
is that of complete locally convex Hausdorff spaces,
just as in \cite{hochmost} and \cite{mackeigh}.
We do not make this precise,
to avoid an orgy of details related with topological vector spaces
and leave the requisite details to the reader.
We only mention here that,
as on p.~376 of \cite{hochmost}, given e.~g. the topological vector space $A$
and the Lie group $G$, we topologize
$C^{\infty}(G,A)$ by taking for a fundamental 
system of neighborhoods of the origin of $C^{\infty}(G,A)$ 
the sets of the kind $N(C,D,U)$ where
$C$ ranges over the compact subsets of $G$, the constituent $D$
over the finite sets of differential operators on $G$, the constituent
$U$ over the neighborhoods of $0$ in $A$, and
$f$ being a member of $N(C,D,U)$ means that $\delta(f)(C)\subseteq U$
for every $\delta \in C$.
In this manner we topologize, in particular, a space of the kind
$C^{\infty}(G,\Gamma(\zeta))$ where $\zeta$ is a smooth vector bundle
and $\Gamma(\zeta))$ its space of smooth sections, suitably topologized.

We will use the notation Ext and Tor for the corresponding infinite 
sequence of relative derived functors, not just for the first derived
functors. This convention will be in force throughout.
 
In equivariant de Rham theory, there is a certain dichotomy between
left and right actions which creates some minor technical difficulties:
In the standard formalism, when a Lie group $G$ acts on a smooth manifold
$X$ from the left, the naturally induced $G$-action on the algebra
of smooth functions on $X$ 
or, more generally, on the de Rham complex of $X$,
is from the right
as is the induced infinitesimal action of the Lie algebra $\fra g$
of $G$. As a side remark, we note that 
this fact has actually created some confusion in the literature.
The formally appropriate approach to equivariant de Rham theory
in terms of the monad 
technology involves
right modules. 
This is the reason why, in \cite{koszultw}, we have built
the theory systematically for right $G$-modules.
However, when the Lie groupoid $\GGG$ 
with base manifold $\BBB$ acts on the fiber bundle $f\colon M \to \BBB$
from the left, the correct way to proceed is to
build the theory in such a way that 
on the object, corresponding to what was the algebra of functions
before or more generally the de Rham complex,
the induced $\GGG$-action is still from the left. 
In terms of the earlier language of an action of a Lie group
on a smooth manifold from the left, this simply means that we switch from
the induced right action on the functions or on the de Rham complex
to a left action in the standard way, an operation
which is always possible since, for a group, 
the group algebra is a Hopf algebra
in a canonical manner. This kind of switch is no longer
possible over a general groupoid, though, and there is a single
consistent procedure to build the theory,
either left-handed or right-handed, and we proceed
to built it left-handed; suffice it to mention that, had we decided to 
build the theory right-handed, we would among others have to write
differential operators on the right of the objects these operators apply to.
In particular, right modules over a Lie-Rinehart algebra
do not naively correspond to left modules; see \cite{bv}
where the situation is discussed in detail.
Piecing the various
items correctly  together and isolating the appropriate monads and categories
is, then, a somewhat laborious puzzle.

In the literature, 
the $G$-equivariant cohomology  $\mathrm H^*_G(\lambda)$
of a Lie algebroid $\lambda \colon E \to B$
acted upon by a compact Lie group $G$ has been explored,
cf. e.~g. \cite{brchroru}.
This theory is substantially 
different from ours, though, which is a theory where elements 
of a  Lie groupoid and of a Lie algebroid 
are viewed as {\em operators on a family of spaces\/}
that give rise to a corresponding equivariant theory 
whereas in equivariant Lie algebroid cohomology the group $G$
acts on the Lie algebroid and hence on its 
associated differential graded algebra of forms.

I am indebted to the referees whose remarks made me clarify
a number of issues and to K. Mackenzie for discussions related with 
his extension of the Hochschild-Mostow 
cohomology theory for topological groups \cite{hochmost} 
to locally trivial topological groupoids
and in particular with his construction of (relatively) injective
modules over a locally trivial topological groupoid \cite{mackeigh}.
I owe a special debt of gratitude to Jim Stasheff for a number
of comments which helped improve the exposition.

\section{Monads and relative cohomology}\label{monads}

Before going into detail
we note that we use the monad-comonad terminology 
exclusively
rather than the equivalent terminology \lq\lq
triple\rq\rq\ etc. 

Let $\mathcal M$ be a category, $\mathcal T\colon
\mathcal M \to \mathcal M$ an endofunctor, 
let $\mathcal I$ denote the identity functor of $\mathcal M$,
and let $\uuu\colon \mathcal I \to  \mathcal T$  
and
$
\mu \colon 
\mathcal T^2 \longrightarrow \mathcal T
$
be natural transformations.
Recall that the triple $(\mathcal T,\uuu, \mu)$ is said to be a {\em monad\/}
over $\mathcal M$ provided the following two diagrams commute:
\begin{equation*}
\begin{CD}
\mathcal T
@<{\mu}<<
\mathcal T^2
@.\phantom{xxxxxxx}@.
\mathcal T
@>=>>
\mathcal T
@<=<<
\mathcal T
\\
@A{\mu}AA
@A{\mathcal T\mu}AA
@.
@A{=}AA
@A{\mu}AA
@A{=}AA
\\
\mathcal T^2
@<{\mu \mathcal T}<<
\mathcal T^3
@.\phantom{xxxxxxx}@.
\mathcal I\mathcal T
@>{\uuu \mathcal T}>>
\mathcal T^2
@<{\mathcal T\uuu}<<
\mathcal T\mathcal I
\end{CD}
\end{equation*}
The natural transformations $\uuu$ and $\mu$ are referred to as the
{\em unit\/} and {\em composition\/}, respectively, of the monad.

The  
{\em dual standard construction\/} associated with the monad $(\mathcal T,\uuu, \mu)$,
cf. \cite{duskinon}, \cite{godebook} (\lq\lq construction fondamentale\rq\rq\ 
on p.~271), \cite{maclbotw},
 yields the cosimplicial object
\[
\left(\mathcal T^{n+1},  \varepsilon^j\colon \mathcal T^{n+1} \to \mathcal T^{n+2},
\eta^j\colon
\mathcal T^{n+2} \to \mathcal T^{n+1}
\right)_{n \in \mathbb N};
\]
here, for $n \geq 0$,
\begin{align}
\varepsilon^j&= \mathcal T^j \uuu \mathcal T^{n-j+1}\colon \mathcal T^{n+1} \to \mathcal T^{n+2},
\ j = 0, \ldots, n+1,
\\
\eta^j&= \mathcal T^j \mu \mathcal T^{n-j}\colon \mathcal T^{n+2} \to \mathcal T^{n+1},
\ j = 0, \ldots, n.
\end{align}
Thus, given the object $V$ of $\mathcal M$, the system
\[
\mathbf T(V)=\left(\mathcal T^{n+1}(V),\varepsilon^j,\eta^j\right)_{n \in \mathbb N}
\]
is a cosimplicial object in $\mathcal M$;
here we do not distinguish in notation
between the natural transformations  $\eta^j$ and $\varepsilon^j$ 
and the morphisms they induce after evaluation of
the corresponding functors in an object.
The dual notion of {\em comonad\/} will be reproduced
in Section \ref{comonads} below.

An adjunction is well known to determine a monad (and a comonad) 
\cite{maclbotw}:
Let $\Cat$ and $\mathcal M$ be categories, let $\mathcal G \colon \Cat \to \mathcal M$
be a functor, suppose that the functor $\square\colon \mathcal M \to \Cat$ 
is left-adjoint to $\mathcal G$, and let 
$$
\mathcal T= \mathcal G \square\colon \mathcal M \longrightarrow \mathcal M.
$$
Let $\mathcal I$ denote the identity functor,
let $\uuu\colon \mathcal I \to  \mathcal T$ be the {\it unit\/},
$\ccc \colon   \square \mathcal G \to \mathcal I$ the {\it counit\/}
of the adjunction, and let $\mu$ be the natural transformation
$$
\mu = \mathcal G \ccc \square\colon\mathcal G \square\mathcal G \square= 
\mathcal T^2 \longrightarrow \mathcal T =\mathcal G \square.
$$
The data $(\mathcal T,\uuu, \mu)$ constitute a {\it monad\/}
over the category $\mathcal M$.
The associated dual standard construction
then defines the relative
$\mathrm{Ext}_{(\mathcal M,\Cat)}$-functor, that is, the $\mathrm{Ext}$-functor
in the category $\mathcal M$ relative to the category $\Cat$.

\begin{Example}\label{ex1}
{\rm Let $R$ be a commutative ring,
$G$ a group,
$\mathrm{Mod}_R$ the category of $R$-modules,
$\mathrm{Mod}_{RG}$ that of right $RG$-modules,
let
\[
\mathcal G\colon \mathrm{Mod}_{R}\longrightarrow \mathrm{Mod}_{RG}
\]
be the familiar functor which assigns to the $R$-module $V$
the right $RG$-module $\mathrm{Map}(G, V)$, with right $G$-structure
being given by left translation in $G$, and let
\[
\square \colon \mathrm{Mod}_{RG}\longrightarrow \mathrm{Mod}_{R}
\]
be the forgetful functor.
The unit of the resulting adjunction is well known to be given by
the assignment to the right $RG$-module $V$ of
\[
\uuu_V\colon V \longrightarrow \mathrm{Map}(G, \square V),
\ v \longmapsto \uuu_v:G \to  \square V,\ \uuu_v(x) =vx, \ v \in V, x \in
G.
\]
Likewise it is a standard fact that 
the counit of the adjunction is given by
the assignment to the $R$-module $W$ of
\[
\ccc_W\colon \square \mathrm{Map}(G, W) \longrightarrow W,
\ h \longmapsto h(e),\ h\colon G \to W.
\]
Consequently the natural transformation
\[
\mu 
\colon \mathcal T^2 \longrightarrow \mathcal T 
\]
is given by the assignment to a right $RG$-module $V$ of the association
\[
\mu_V 
\colon \mathrm{Map}(G,\square\mathrm{Map}(G,\square V))
\longrightarrow \mathrm{Map}(G,\square V),\ 
(\mu_V (\Phi))(x)= \Phi(x)(e),\ x \in G,
\] 
where $\Phi$ ranges over maps of the kind $G \to \mathrm{Map}(G,\square V)$.
Given the right $RG$-module $V$,
the dual standard construction 
associated with $V$ and the resulting monad
$(\mathcal T,\uuu,\mu)$ 
yields 
the cosimplicial object 
\begin{equation}
\mathbf T(V)=\left(\mathrm{Map}(G^{\times (n+1)}, 
\square V),\varepsilon^j,\eta^j\right)_{n \in \mathbb N}
\label{T}
\end{equation}
in the category of right $RG$-modules,
and the associated chain complex $|\mathbf T(V)|$,
together with the injection $\uuu_V$,
is an injective resolution of $V$ in the category
of  right $RG$-modules. Given the right $RG$-module $W$,
the chain complex $\mathrm{Hom}_G(|\mathbf T(V)|,W)$ then defines
$\mathrm{Ext}_G(V,W)$.
For intelligibility we recall that the right $RG$-module structure
on the degree $n$ constituent
$\mathrm{Map}(G^{\times (n+1)}, 
\square V)$
is given by the association
\begin{equation}
\begin{aligned}
\mathrm{Map}(G^{\times (n+1)}, 
\square V)
\times G 
&\longrightarrow
\mathrm{Map}(G^{\times (n+1)}, 
\square V),\ (\alpha,x) \mapsto \alpha \cdot x,
\\
(\alpha \cdot x)(x_0,x_1,\ldots,x_n) &=\alpha (xx_0,x_1,\ldots,x_n),
\end{aligned}
\label{action1}
\end{equation}
where $\alpha \in \mathrm{Map}(G^{\times (n+1)}, 
\square V)$ and $x,x_0,\ldots, x_n \in G$.
All this is entirely standard and classical.}
\end{Example}

\begin{Example}\label{ex2}
{\rm This example is a variant of Example \ref{ex1},
for {\em Lie groups} rather than just discrete groups
and {\em smooth maps\/} rather than 
just set maps:
Let the ground ring be that of the reals, $\mathbb R$, and
let $G$ be a Lie group.
We will use the notion of differentiable $G$-module 
in the  sense of \cite{hochmost}, and we will interpret
the notation \lq\lq Map\rq\rq\ in the previous example as $C^{\infty}$, 
that is, as ordinary smooth maps.
Let $\Cat=\mathrm{Vect}$, the category of real vector spaces,
$\mathcal M=\mathrm{Mod}_{G}$ that of differentiable right $G$-modules, and let
$\mathcal G_G \colon \mathrm{Vect} \to  \mathrm{Mod}_{G}$
be the functor
which assigns to the real vector space
$V$ the 
$G$-representation
\[
\mathcal G_GV= C^{\infty}(G, V),
\]
endowed with the right $G$-module structure coming from 
left translation in $G$.
We use the font $\mathcal G$ merely for convenience since this is reminiscent
of the notation $G$ in \cite{maclbotw} 
for this kind of functor; this usage of the font $\mathcal G$
has nothing to do with our usage of the notation $G$ 
for the group variable.
The functor $\mathcal G_G$
is right adjoint to the forgetful functor $\square 
\colon \mathrm{Mod}_{G} \to  \mathrm{Vect}$
and hence defines a monad  $(\mathcal T,\uuu,\mu)$  over the category 
$\mathrm{Mod}_{G}$.
Given the differentiable right $G$-module $V$, 
the
{\em dual standard construction\/}
\begin{equation}
\mathbf T(V)=\left(C^{\infty}(G^{\times (n+1)}, 
\square V),\varepsilon^j,\eta^j\right)_{n \in \mathbb N}
\label{Tsmooth}
\end{equation}
associated with $V$  and 
the monad $(\mathcal T,\uuu,\mu)$
is a cosimplicial object in the category of differentiable right $G$-modules,
and the associated chain complex $|\mathbf T(V)|$,  together with the 
injection $\uuu_V$,
is a differentiably injective resolution of $V$
in the category of differentiable right $G$-modules;
indeed,
apart from the fact that we deal with 
right $G$-modules rather 
than left $G$-modules, this resolution is exactly one
of the kind considered on p.~369 of \cite{hochmost}.
Given the differentiable right $RG$-module $W$,
the chain complex $\mathrm{Hom}_G(|\mathbf T(V)|,W)$ then defines
the differentiable 
\[
\mathrm{Ext}_G(V,W)=\mathrm{Ext}_{(\mathcal M,\Cat)}(V,W),
\]
that is, the Ext in the category of differentiable right $G$-modules
relative to the category of ordinary real vector spaces.
The resolution $|\mathbf T(V)|$ of $V$
defines, in particular, the differentiable cohomology 
of $G$ with values in $V$
in the sense of \cite{hochmost}, and we refer to that paper
for details.
}
\end{Example}

\begin{Example}\label{ex3}
{\rm To prepare for our ultimate goal to 
develop an equivariant cohomology theory 
with respect to a Lie groupoid, we will now give an alternate
description of the monad in
Example \ref{ex1}: Thus, let $R$ be an arbitrary commutative ground ring,
let $\mathcal U\colon \mathrm{Mod}_{RG}\longrightarrow \mathrm{Mod}_{RG}$
be the functor which assigns to the right $RG$-module $V$ the
right $RG$-module $\mathcal U(V)=\mathrm{Map}(G,V)$,
endowed with the right {\em diagonal\/} action
\begin{equation}
\mathrm {Map}(G,V)\times G \longrightarrow \mathrm{Map}(G,V),\ 
(\rhoo,x)\longmapsto \rhoo \cdot x
\label{diag22}
\end{equation}
given by
\[
(\rhoo\cdot x)(y) = (\rhoo(xy))x,\ x, y \in G, \rhoo \colon G \to V,
\]
let $\omega$ be the natural transformation given by
the assignment to the right $RG$-module $V$ of
\[
\omega=\omega_V\colon V \longrightarrow \mathrm{Map}(G, V),
\ v \longmapsto \omega_v:G \to V,\ \omega_v(x) =v, \ v \in V, x
\in G,
\]
and let
\[
\nu\colon \mathcal U^2 \longrightarrow \mathcal U
\]
be 
the natural transformation
given by the assignment to
a right $RG$-module $V$ of the association
\[
\nu_V 
\colon \mathrm{Map}(G,\mathrm{Map}(G,V))
\longrightarrow \mathrm{Map}(G,V),\ 
(\nu_V (\Psi))(x)= \Psi(x)(x),\ x \in G,
\] 
where $\Psi$ ranges over maps of the kind $G \to \mathrm{Map}(G,V)$.
The system $(\mathcal U,\omega, \nu)$
is a monad over the category $\mathrm{Mod}_{RG}$,
indeed, an alternate description of the monad $(\mathcal T,\uuu,\mu)$
in Example \ref{ex1}, as we will show shortly.
The following is well known and classical.

\begin{Proposition} \label{prop11} Let $V$ be a
right $G$-module. Relative to the right $G$-module structures
on $\mathrm{Map}(G,V)$ and  $\mathrm{Map}(G,\square V)$, 
the map
\begin{equation}
\vartheta=\vartheta_V\colon 
\mathrm{Map}(G,V)\longrightarrow \mathrm{Map}(G,\square V)
\label{phii}
\end{equation}
given by
\[
(\vartheta_V(\rhoo))(y)=(\rhoo(y))y
\]
is a natural isomorphism of right $G$-modules, and
\begin{equation}
\begin{CD}
V@>{\omega_V}>> \mathrm{Map}(G,V) 
\\
@V{\mathrm{Id}}VV
@VV{\vartheta_V}V
\\
V @>>{\uuu_V}> \mathrm{Map}(G,\square V)
\end{CD}
\label{CD11}
\end{equation}
is a commutative diagram in the category of right $G$-modules.
Furthermore, 
the diagram
\begin{equation}
\begin{CD}
\mathrm{Map}(G,\mathrm{Map}(G,V))
@>{\nu_V}>> \mathrm{Map}(G, V)
\\
@V{\mathrm{Map}(G,\vartheta_V)}VV
@V{\vartheta_V}VV
\\
\mathrm{Map}(G,\mathrm{Map}(G,\square V))
@. \mathrm{Map}(G, \square V)
\\
@V{\vartheta_{\mathrm{Map}(G,\square V)}}VV
@V{\mathrm{Id}}VV
\\
\mathrm{Map}(G,\square\mathrm{Map}(G,\square V))
@>{\mu_V}>> \mathrm{Map}(G,\square V)
\end{CD}
\label{CD2}
\end{equation}
is commutative.
\end{Proposition}

\begin{proof}
Indeed, given $y \in G$ and  $\rhoo \in \mathrm{Map}(G,V)$,
\[
(\vartheta_V \rhoo)(y) = ((\vartheta_V \rhoo)\cdot y) (e) =
(\vartheta_V (\rhoo \cdot y))(e) =  (\rhoo \cdot y)(e) =
(\rhoo (y))y .
\]
Consequently the diagram \eqref{CD11} is commutative.
Furthermore,
given $\Psi \colon G \to \mathrm{Map}(G,V)$ and $x \in G$,
\begin{align*}
(\vartheta_V(\nu_V \Psi))(x) &= ((\Psi(x))(x))x
\\
(\mu_V(\vartheta_{\mathrm{Map}(G,\square V)}(\mathrm{Map}(G,\vartheta_V))(\Psi)))(x)
&=(\vartheta_{\mathrm{Map}(G,\square V)}(\mathrm{Map}(G,\vartheta_V))(\Psi))(x)(e)
\\
&
=(\mathrm{Map}(G,\vartheta_V)(\Psi))(x)(x)
\\
&=(\vartheta_V(\Psi(x)))(x)
\\
&= ((\Psi(x))(x))x
\end{align*}
whence the diagram \eqref{CD2} is commutative.
\end{proof}

\noindent
\begin{Remark}
{\rm The isomorphism $\vartheta_V$ is determined by the requirement
\[
(\vartheta_V\rhoo)(e) = \rhoo(e),\ \rhoo \in \mathrm{Map}(G,V) .
\]
Furthermore, on both $\mathrm{Map}(G,V)$ and $\mathrm{Map}(G,\square V)$
the right $G$-module structure relative to the copy of $G$ comes from
left translation in $G$.}
\end{Remark}

\begin{Corollary}\label{cor1}
The natural transformation $\vartheta$ in Proposition {\rm \ref{prop11}}
yields a natural isomorphism of monads from
the monad
$(\mathcal U,\omega, \nu)$ 
to the monad
$(\mathcal T, \uuu, \mu)$ in Example {\rm \ref{ex1}}.
\end{Corollary}

\begin{Remark} This isomorphism between the two monads 
is presumably folk-lore and relies on the fact, well understood in the classical literature,
cf. e.~g. p.~{\rm 212} of {\rm \cite{hilstatw}},
that the diagonal map of the group $G$ turns the group ring $RG$ of $G$ into a Hopf algebra.
\end{Remark}

Let $V$ be a right $RG$-module.
The dual standard construction 
\begin{equation}
\mathbf U(V)=\left(\mathrm{Map}(G^{\times (n+1)}, 
V),\varepsilon^j,\eta^j\right)_{n \in \mathbb N}
\label{U}
\end{equation}
associated with $V$ and the monad
$(\mathcal U,\omega,\nu)$  
is a cosimplicial object in the category of right $RG$-modules.
For intelligibility, we recall
the cosimplicial structure; on the degree $n$ constituent
$\mathrm{Map}(G^{\times (n+1)},V)$, this structure
 is given by the familiar formulas
\begin{align*}
(\varepsilon^j(\varphi))(x_0,\ldots,x_{j-1},x_j,x_{j+1},\ldots,x_n)
&=\varphi(x_0,\ldots,x_{j-1},x_{j+1},\ldots,x_n),\ 0 \leq j \leq n,
\\
(\eta^j(\varphi))(x_0,\ldots,x_{n-1})
&=\varphi(x_0,\ldots,x_{j-1},x_j,x_j,\ldots,x_{n-1}),\ 0 \leq j \leq n-1.
\end{align*}
Likewise, 
on the degree $n$ constituent
$\mathrm{Map}(G^{\times (n+1)}, 
V)$, the right $RG$-module structure
is given by the association
\begin{equation}
\begin{aligned}
\mathrm{Map}(G^{\times (n+1)}, 
V)
\times G 
&\longrightarrow
\mathrm{Map}(G^{\times (n+1)}, 
V),\ (\alpha,x) \mapsto \alpha \cdot x,
\\
(\alpha \cdot x)(x_0,x_1,\ldots,x_n) &=(\alpha (xx_0,xx_1,\ldots,xx_n))x,
\end{aligned}
\label{action2}
\end{equation}
where $\alpha \in \mathrm{Map}(G^{\times (n+1)}, 
\square V)$ and $x,x_0,\ldots, x_n \in G$.
Together with the injection $\omega_V$,
the resulting chain complex $|\mathbf U(V)|$
is an injective resolution of $V$ in the category
of  right
$RG$-modules.

\begin{Proposition} \label{prop2}
The natural isomorphism of monads mentioned in Corollary
{\rm \ref{cor1}} induces a natural isomorphism
\begin{equation*}
\Theta\colon \mathbf U(V)=\left(\mathrm{Map}(G^{\times (n+1)}, 
V),\varepsilon^j,\eta^j\right)_{n \in \mathbb N}
\longrightarrow
\left(\mathrm{Map}(G^{\times (n+1)}, 
\square V),\varepsilon^j,\eta^j\right)_{n \in \mathbb N}=
\mathbf T(V)
\end{equation*}
of cosimplicial objects in the category of right $RG$-modules
and hence a natural isomorphism between the resulting
injective resolutions of $V$ in the category of right $RG$-modules.
In degree $n \geq 0$, this isomorphism is given by the association
\begin{equation}
\begin{aligned}
\vartheta_n\colon \mathrm{Map}(G^{\times (n+1)}, V)
&\longrightarrow 
\mathrm{Map}(G^{\times (n+1)}, \square V)
\\
\vartheta_n(\alpha)(x_0,x_1,\ldots, x_n) &= 
\alpha (x_0,x_0x_1,\ldots, x_0x_1\ldots x_{n-1}, x_0x_1\ldots x_n)\cdot x_0\cdot \ldots \cdot x_n
\end{aligned}
\label{ass1}
\end{equation}
where $x_0,x_1,\ldots x_n \in G$ and 
$\alpha \in \mathrm{Map}(G^{\times (n+1)}, V)$.
\end{Proposition}

\begin{proof} Only the statement making the isomorphism explicit
requires proof. For $n=0$, the given formula comes down to the commutativity
of diagram \ref{CD11} above. For $n \geq 1$,
the verification is straightforward and left to the reader. A 
conceptual explanation will be given in Section \ref{comonads} below.
\end{proof}
}
\end{Example}

\begin{Example}\label{ex4}
{\rm  This example is the variant of Example \ref{ex3}
for {\em Lie groups} rather than just discrete groups
and {\em smooth maps\/} rather than 
just set maps; it is related to Example \ref{ex2} in the same way as
Example \ref{ex3} is related to Example \ref{ex1}:
Let the ground ring be that of the reals, $\mathbb R$, and let $G$ be a 
Lie group. Further, let 
$\mathcal U\colon \mathrm{Mod}_{G}\longrightarrow \mathrm{Mod}_{G}$
be the functor which assigns to the right $G$-representation $V$ the
right $G$-representation $\mathcal U(V)=C^{\infty}(G,V)$,
endowed with the right {\em diagonal\/} action
\begin{equation}
C^{\infty}(G,V)\times G \longrightarrow C^{\infty}(G,V),\ 
(\rhoo,x)\longmapsto \rhoo \cdot x
\label{diag222}
\end{equation}
given by
\begin{equation}
(\rhoo\cdot x)(y) = (\rhoo(xy))x,\ x, y \in G, \rhoo \colon G \to V,
\label{223}
\end{equation}
let $\omega$ be the natural transformation given by
the assignment to the right $G$-representation $V$ of
\begin{equation}
\omega=\omega_V\colon V \longrightarrow C^{\infty}(G, V),
\ v \longmapsto \omega_v:G \to V,\ \omega_v(x) =v, \ v \in V, x
\in G,
\label{224}
\end{equation}
and let
\begin{equation}
\nu\colon \mathcal U^2 \longrightarrow \mathcal U
\label{nu}
\end{equation}
be 
the natural transformation
given by the assignment to
the right $G$-representation $V$ of the association
\[
\nu_V 
\colon C^{\infty}(G,C^{\infty}(G,V))
\longrightarrow C^{\infty}(G,V),\ 
(\nu_V (\Psi))(x)= \Psi(x)(x),\ x \in G,
\] 
where $\Psi$ ranges over smooth maps of the kind $G \to C^{\infty}(G,V)$.
The system $(\mathcal U,\omega, \nu)$
is a monad over the category $\mathrm{Mod}_{G}$,
indeed, an alternate description of the monad $(\mathcal T,\uuu,\mu)$
in Example \ref{ex2}.
More precisely, as in the situation of Example \ref{ex3},
the natural transformation $\vartheta$ which, to the right $G$-representation 
$V$,
assigns the isomorphism
\begin{equation}
\vartheta_V\colon C^{\infty}(G,V)\longrightarrow C^{\infty}(G,\square V)
\label{phi}
\end{equation}
of right $G$-modules
given by
\[
(\vartheta(\rhoo))(y)=(\rhoo(y))y
\]
yields an isomorphism of monads from
the monad
$(\mathcal U,\omega, \nu)$ 
to the monad
$(\mathcal T, \uuu, \mu)$ in Example {\rm \ref{ex2}}.

Given the differentiable right $G$-module $V$, 
the 
{\em dual standard construction\/}
\begin{equation}
\mathbf U(V)=\left(C^{\infty}(G^{\times (n+1)}, 
V),\varepsilon^j,\eta^j\right)_{n \in \mathbb N}
\label{UVsmooth}
\end{equation}
associated with $V$  and the monad
$(\mathcal U,\omega,\nu)$ is a cosimplicial object in the category
of differentiable right $G$-modules,
and the chain complex $|\mathbf U(V)|$ associated with 
this cosimplicial object,
together with the injection $\omega_V$,
is a differentiably injective resolution of $V$
in the category of differentiable right $G$-modules.
Apart from the fact that we deal with differentiable
right $G$-modules rather than differentiable left $G$-modules,
the resolution  $|\mathbf U(V)|$ is, in fact, precisely the 
{\em homogeneous resolution\/} of $V$, cf. p.~371 
of \cite{hochmost}.

The association \eqref{ass1}, now interpreted in the smooth category,
that is, with $C^{\infty}$ substituted for $\mathrm{Map}$,
yields the isomorphism
\[
|\Theta|\colon |\mathbf U(V)|\longrightarrow |\mathbf T(V)|
\]
of differentiably injective resolutions of $V$
in the category
of differentiable right $G$-modules,
where $\mathbf T(V)$ is the cosimplicial object \eqref{Tsmooth}.
}
\end{Example}

\begin{Example} \label{ex7}
{\rm 
Let $R$ be an arbitrary commutative ring with 1 and $\mathfrak g$ 
an $R$-Lie algebra, which we suppose to be projective as an $R$-module.
Let $C\fra g$ be the cone on $\fra g$ 
in the category of  differential 
graded Lie algebras; the cone $C\fra g$ is a contractible differential 
graded $R$-Lie algebra. 
Let $\Cat=\chain_{\mathfrak g}$, the category of right $\mathfrak g$-chain 
complexes,
let $\mathcal M=\mathrm{Mod}_{C\mathfrak g}$,
and let $\mathcal G^{\mathfrak g}_{C\mathfrak g} \colon\chain_{\mathfrak g}\to \mathrm{Mod}_{C\mathfrak g}$
be the functor given by
\begin{equation}
\mathcal G^{\mathfrak g}_{C\mathfrak g}(V)= \mathrm{Hom}_{\mathfrak g}(\mathrm U[C\mathfrak g],V)
\cong
\mathrm{Hom}^{\tau_{\mathfrak g}}(\Lambda'_{\partial}[s\mathfrak g],V) \cong
(\mathrm{Alt}(\mathfrak g, V),d),
\label{2.3.1}
\end{equation}
the total object arising from the bicomplex having 
$\mathrm{Alt}^*(\mathfrak g, V_*)$ as underlying bigraded $R$-module;
here $V$ ranges over  right $\mathfrak g$-chain complexes,
$(\mathrm{Alt}(\mathfrak g, V),d)$ is
endowed with the obvious right $(C\mathfrak g)$-module structure coming from
the obvious left $(\mathrm U[C\mathfrak g])$-module structure on itself or,
equivalently, that given by the operations of 
contraction and Lie derivative on the CCE complex
$(\mathrm{Alt}(\mathfrak g, V),d)$, cf. (1.3) above.
The functor $\mathcal G^{\mathfrak g}_{C\mathfrak g}$
is right adjoint to the forgetful functor $\square 
\colon \mathrm{Mod}_{C\mathfrak g}\to \chain_{\mathfrak g}$
and hence defines a monad  $(\mathcal T,\uuu,\mu)$  over the category 
$\mathrm{Mod}_{C\mathfrak g}$.
Given the right $(C\mathfrak g)$-module $\VV$, 
the chain complex $\left|\mathbf T(\VV)\right|$ arising from the
{\em dual standard construction\/}
$\mathbf T(\VV)$ associated with $\VV$
is a resolution of $\VV$
in the category of $(C\mathfrak g)$-modules that is injective relative to
the category $\chain_{\mathfrak g}$ of right $\mathfrak g$-chain complexes.
Given a right  $(C\mathfrak g)$-module $\mathbf W$, the {\em relative differential 
$\mathrm{Ext}_{(C \mathfrak g, \mathfrak g)}(\mathbf W,\VV)$\/} is the homology of
the chain complex
$$
\mathrm{Hom}_{C\mathfrak g}\left(\mathbf W,\left|\mathbf T(\VV)\right|\right).
$$
In particular, for $\mathbf W=R$, 
the relative differential graded 
$\mathrm{Ext}_{(C \mathfrak g, \mathfrak g)}(R,\VV)$ is the homology of
the chain complex
$
\left|\mathbf T(\VV)\right|^{C\mathfrak g}.
$
}
\end{Example}

\begin{Example}\label{ex9}
{\rm  This example extends Example \ref{ex4}.
Let the ground ring again be that of the reals, $\mathbb R$, and let $G$ be a 
Lie group. We will now switch to differentiable {\em left\/} $G$-modules.
Let $\xi\colon P \to B$ be a right principal $G$-bundle.
Let $\mathcal U_{\xi}\colon {}_G\mathrm{Mod}\longrightarrow {}_G\mathrm{Mod}$
be the functor which assigns to the left $G$-representation $V$ the
left $G$-representation $\mathcal U_{\xi}(V)=C^{\infty}(P,V)$,
endowed with the left {\em diagonal\/} action
\begin{equation}
G\times C^{\infty}(P,V) \longrightarrow C^{\infty}(P,V),\ 
(x,\rhoo)\longmapsto x\cdot \rhoo 
\label{diag2222}
\end{equation}
given by
\begin{equation}
(x\cdot \rhoo)(y) = x(\rhoo(yx)),\ x\in G, y \in P, \rhoo \colon P \to V,
\label{2223}
\end{equation}
let $\omega$ be the natural transformation given by
the assignment to the left $G$-representation $V$ of
\begin{equation}
\omega=\omega_V\colon V \longrightarrow C^{\infty}(P, V),
\ v \longmapsto \omega_v:P \to V,\ \omega_v(y) =v, \ v \in V, y
\in P,
\label{2224}
\end{equation}
and let
\begin{equation}
\nu\colon \mathcal U_{\xi}^2 \longrightarrow \mathcal U_{\xi}
\label{nnu}
\end{equation}
be 
the natural transformation
given by the assignment to
the left $G$-representation $V$ of the association
\[
\nu_V 
\colon C^{\infty}(P,C^{\infty}(P,V))
\longrightarrow C^{\infty}(P,V),\ 
(\nu_V (\Psi))(y)= \Psi(y)(y),\ y \in P,
\] 
where $\Psi$ ranges over smooth maps of the kind $P \to C^{\infty}(P,V)$.
The system $(\mathcal U_{\xi},\omega, \nu)$
is a monad over the category ${}_G\mathrm{Mod}$.

When $P$ is just $G$, so that $\xi$ is the trivial principal $G$-bundle
over a point, the monad $(\mathcal U_{\xi},\omega, \nu)$
plainly comes down to
the monad
$(\mathcal U,\omega, \nu)$
spelled out in Example \ref{ex3},
phrased in the category of differentiable left $G$-modules
rather than  differentiable right $G$-modules.

Given the differentiable left $G$-module $V$, 
the 
{\em dual standard construction\/}
\begin{equation}
\mathbf U_{\xi}(V)=\left(C^{\infty}(P^{\times (n+1)}, 
V),\varepsilon^j,\eta^j\right)_{n \in \mathbb N}
\label{UVsmoothxi}
\end{equation}
associated with $V$  and the monad
$(\mathcal U_{\xi},\omega,\nu)$ is a cosimplicial object in the category
of differentiable left $G$-modules,
and the chain complex $|\mathbf U_{\xi}(V)|$ associated with 
this cosimplicial object,
together with the injection $\omega_V$,
is a differentiably injective resolution of $V$
in the category of differentiable left $G$-modules.
In particular, the obvious {\em restriction\/} map
from $|\mathbf U_{\xi}(V)|$ to 
the chain complex $|\mathbf U(V)|$ associated with 
the cosimplicial object 
\eqref{UVsmooth} but phrased for differentiable
left $G$-modules rather than differentiable right ones,
is a comparison map between the two resolutions.
For the special case where $\xi$ is the trivial principal bundle,
the obvious map from
$|\mathbf U(V)|$
to $|\mathbf U_{\xi}(V)|$  
is a comparison map in the other direction.
In the general case, using (i) an open cover of $B$
such that $\xi$ is trivial on each member of the cover
and, furthermore, a smooth partition of unity subordinate
to this cover, we can still construct
a comparison map from $|\mathbf U(V)|$
to $|\mathbf U_{\xi}(V)|$.

For intelligibility we recall that, given a discrete group $\pi$ and 
a $\pi$-module $V$, for any free $\pi$-set $\Gamma$,
the injection
\[
V \longrightarrow \mathrm{Map}(\Gamma,V),\ v \longmapsto \varphi_v,
\varphi_v(y)= v, v\in V, y \in \Gamma,
\]
is one of $V$ into a relatively injective $\pi$-module,
$\mathrm{Map}(\Gamma,V)$ being suitable turned into a $\pi$-module.
For our purposes, $C^{\infty}(P,V)$ has formally the same significance
as a relatively injective $\pi$-module of the kind
$\mathrm{Map}(\Gamma,V)$.}
\end{Example}

\section{Differentiable cohomology over a Lie groupoid}
\label{liegroupoids}

We shall see that the monads spelled out in
Example \ref{ex3} and Example \ref{ex4}
(and written as $(\mathcal U,\omega,\nu)$)
generalize to groupoids and yield the appropriate injective resolutions
while the monads spelled out in
Example \ref{ex1} and Example \ref{ex2} do not extend 
in an obvious manner  to groupoids. See also Remark \ref{cd5} below.

Our reference for  terminology is \cite{canhawei}, 
\cite{mackone}--\cite{mackbtwo};
we will use the conventions in \cite{canhawei}.
Let $\GGG$ be a Lie groupoid; we denote the smooth manifold of {\em objects\/} 
by $\BBB$,
the {\em source\/} and {\em target\/} maps by $\sou\colon \GGG\to \BBB$ and
$\tar\colon \GGG\to \BBB$, respectively, and the {\em object inclusion\/}
map by $1\colon \BBB \to \GGG$.
Thus we view the element $u$ of $\GGG$ as an arrow from
$\sou(u)$ to $\tar(u)$.
According to the conventions in \cite{canhawei} and \cite{mackeigh},
the elements $u,v$ of $\GGG$ are {\em composable\/} when
$s(u)=t(v) \in \BBB$, and we then write the element of $\GGG$
arising from composition simply as $uv$.
Often we do not distinguish in notation between the groupoid and its 
space of (iso)morphisms.
Given $p,q \in \BBB$ we will denote the $\tar$-fiber
$\tar^{-1}(p)$ by $\GGG_p$
(beware: in \cite{mackeigh}  the $\tar$-fiber is written as $\GGG^p$), the
$\sou$-fiber
$\sou^{-1}(q)$ by $\GGG^q$ (beware: in \cite{mackeigh} 
 the $\sou$-fiber
is written as $\GGG_q$)
and we write
\[
\GGG_p^q = \GGG_p\cap\GGG^q;
\]
thus $\GGG_q^q$ is the {\em vertex group\/} at $q$.
Moreover, given the objects $p$ and $q$, the groupoid composition amounts to
a smooth map
\begin{equation}
\GGG^q \times \GGG_q \longrightarrow \GGG, \ 
(u,v) \mapsto uv.
\label{comp}
\end{equation}
We suppose throughout that $\GGG$
is {\em locally trivial\/} in the sense that
the smooth map 
\[
(\tar,\sou)\colon \GGG \longrightarrow \BBB \times \BBB
\] 
is
a submersion. Then $\GGG$ is locally trivial
in the usual sense; see p.~16 of \cite{mackbtwo} for details.
Furthermore, for any $p\in \BBB$, the projection $s\colon 
\GGG_p \to \BBB$ is, then, a principal left 
$\GGG_p^p$-bundle and, for any $q\in \BBB$, 
the projection
$t\colon 
\GGG^q \to \BBB$ is a principal right 
$\GGG_q^q$-bundle. 
For any object $q$, the smooth map \eqref{comp} induces a diffeomorphism
from $\GGG^q \times_{\GGG^q_q} \GGG_q$ onto $\GGG$,
and the target map $\tar\colon \GGG \to \BBB$ thus arises as the fiber bundle
\begin{equation}
\tar\colon \GGG^q \times_{\GGG^q_q} \GGG_q \longrightarrow \BBB
\end{equation}
associated with the principal right $\GGG^q_q$-bundle 
$\tar\colon \GGG^q \to \BBB$ and the left $\GGG^q_q$-action
on $\GGG_q$;
likewise the source map $\sou\colon \GGG \to \BBB$ arises  as the fiber bundle
\begin{equation}
\sou\colon \GGG^q \times_{\GGG^q_q} \GGG_q \longrightarrow \BBB
\end{equation}
associated with the principal left $\GGG^q_q$-bundle 
$\sou\colon \GGG_q \to \BBB$ and the right $\GGG^q_q$-action
on $\GGG^q$.

It is well known, cf. \cite{mackbtwo} (Theorem 1.6.5),
that any locally trivial Lie groupoid is the gauge groupoid of
an associated principal bundle:
Let $G$ be a Lie group and
$\xi \colon P \to B$ a principal right $G$-bundle.
The {\em gauge groupoid\/}
$((P\times P)\big/G,B,\sou,\tar,1)$ of $\xi$
arises from the {\em product groupoid\/} or {\em pair groupoid\/}
$(P\times P,P,\sou,\tar,1)$
having composition given by the association
\[
(u,v)(v,w) = (u,w),\ u,v,w \in P
\]
in the obvious way as indicated.
Given the locally trivial Lie groupoid $\GGG$, pick $q \in \BBB$,
let $G=\GGG_q^q$,
and let
\[
\xi =\tar \colon \GGG^q \longrightarrow \BBB;
\]
as noted above, $\xi$ is a principal right $G$-bundle;
the groupoid composition \eqref{comp} can then be written as
\begin{equation}
\GGG^q \times \GGG^q \longrightarrow \GGG, \ 
(u,v) \mapsto uv^{-1},
\label{comp2}
\end{equation}
and this association induces an isomorphism between the gauge groupoid
of $\xi$ and $\GGG$.

Given two spaces $f_1\colon X_1 \to Y$ and  $f_2\colon X_2 \to Y$ 
over the space $Y$,
we will use the notation 
\[
X_1\times_{f_1,Y,f_2}X_2 
=
\{(x_1,x_2); f_1(x_1) = f_2(x_2)\in Y \}\subseteq X_1 \times X_2.
\]
Thus $X_1\times_{f_1,Y,f_2}X_2$ is the fiber product
of $X_1$ and $X_2$ over $Y$.

Let 
$M$ be a smooth manifold and let $f \colon M \to \BBB$ be a 
{\em smooth manifold over\/} $\BBB$, that is, $f$ is 
a smooth map.
A smooth {\em left action\/} of $\GGG$ on $f$, 
cf. \cite{canhawei}\ (p.~101), is given by a commutative diagram
\begin{equation}
\begin{CD}
\GGG \times_{\sou,\BBB,f} M @>>> M
\\
@V{\tar}VV
@VfVV
\\
\BBB @>{\mathrm{Id}}>> \BBB
\end{CD}
\label{action11}
\end{equation}
such that the obvious associativity constraint is satisfied.
We will then say that $f$ is a smooth left $\GGG$-manifold.
The notions of smooth {\em right action\/}
\begin{equation}
\begin{CD}
M \times_{f,\BBB,\tar} \GGG  @>>> M
\\
@V{\sou}VV
@VfVV
\\
\BBB @>{\mathrm{Id}}>> \BBB
\end{CD}
\label{action111}
\end{equation}
of $\GGG$ on $f$
and of smooth right $\GGG$-manifold are defined accordingly.
Morphisms of left $\GGG$-manifolds and morphisms of right 
$\GGG$-manifolds are defined in the obvious way, 
and left $\GGG$-manifolds as well as right $\GGG$-manifolds
constitute a category.

An ordinary group acts on itself by left translation and by right translation.
Accordingly,
left translation in $\GGG$ induces a left action 
of $\GGG$ on 
$f=\tar\colon \GGG \to \BBB$, viewed here as a fiber bundle
having, over $q \in \BBB$, fiber $\GGG_q$.
Indeed, in view the previous discussion, 
after a choice of $q \in \BBB$ has been made,
$\tar\colon \GGG \to \BBB$
acquires the structure
of a fiber bundle associated with
the right principal $\GGG^q_q$-bundle 
$\tar\colon \GGG^q \to \BBB$
and the left translation action of
$\GGG^q_q$ on $\GGG_q$.
In the same vein,
right translation in $\GGG$ induces a right action 
of $\GGG$ on 
$f=\sou\colon \GGG \to \BBB$ viewed as a fiber bundle
having, over $p \in \BBB$, fiber $\GGG^p$;
after a choice of $p \in \BBB$ has been made,
we can view $\sou$ as a fiber bundle associated with
the left principal $\GGG^p_p$-bundle 
$\sou\colon \GGG_p \to \BBB$ via the right translation action of
$\GGG^p_p$ on $\GGG_p$.

In particular, 
let  $\zeta\colon E \to \BBB$ be a vector bundle.
A {\em representation\/}
of $\GGG$ on $\zeta$ 
{\em from the left\/}
or, synonymously,
a {\em left linear action\/} of $\GGG$ on $\zeta$,
is defined in the obvious manner as a left 
action
of $\GGG$ on $\zeta$ that is linear in the 
obvious sense, and we then refer to $\zeta$ as a differentiable {\em left\/}
$\GGG$-module.
Morphisms of left $\GGG$-modules 
are defined in the obvious way, 
and left $\GGG$-modules 
constitute a category which we will denote by
${}_{\GGG}\mathrm{Mod}$.
We do not consider right $\GGG$-modules.

Let $\zeta\colon E \to \BBB$ be a vector bundle endowed with
a left $\GGG$-module structure.
The construction in (2.3) of \cite{mackeigh} yields a differentiably 
injective left $\GGG$-module 
\[
F(\GGGG,\zeta)\colon F(\GGGG,E)\longrightarrow \BBB, 
\]
that is, $F(\GGGG,\zeta)$ is the projection map of
a vector bundle over $\BBB$ together with (i) a left $\GGG$-module
structure 
\begin{equation}
\begin{CD}
\GGG \times_{\sou, \BBB, F(\GGGG,\zeta)} F(\GGGG,E)   @>>>  F(\GGGG,E)
\\
@V{\tar}VV
@VVF(\GGGG,\zeta)V
\\
\BBB @>{\mathrm{Id}}>> \BBB
\end{CD}
\label{action112}
\end{equation}
of $\GGG$ on $F(\GGGG,\zeta)$
and  (ii) a canonical injection 
$\omega_{\zeta} \colon E \to F(\GGGG,\zeta)$
such that the diagram
\begin{equation*}
\begin{CD}
E @>{\omega_{\zeta}}>> F(\GGGG,E)
\\
@V{\zeta}VV
@VV{F(\GGGG,\zeta)}V
\\
\BBB
@>{\mathrm {Id}}>>
\BBB
\end{CD}
\end{equation*}
is a morphism of left $\GGG$-modules; furthermore, $F(\GGGG,\,\cdot \,)$ 
and $\omega$
are natural in the $\GGG$-module variable,
that is,
$F(\GGGG,\,\cdot \,)$ is an endofunctor of ${}_{\GGG}\mathrm{Mod}$
and $\omega$ is a natural transformation 
$\mathcal I \to  F(\GGGG,\,\cdot \,)$.
More precisely, the construction has the following structural properties:

\begin{enumerate}

\item For $q\in \BBB$, the fiber $(F(\GGGG,\zeta))^{-1}(q) 
\subseteq F(\GGGG,E)$
is the space $C^{\infty}(\GGG^q,E_q)$;

\item the embedding $\omega_{\zeta}$ is the {\em constant\/} one in the sense
that, for $q\in \BBB$, the restriction
\[
\omega_{\zeta}|_{E_q}\colon E_q \longrightarrow C^{\infty}(\GGG^q,E_q)
\]
sends $v \in E_q$ to the constant map which assigns $v \in E_q$ to 
$x\in \GGG^q$;

\item the left $\GGG$-action on $F(\GGGG,\zeta)$ is the 
action arising from right translation in $\GGG$ and from the left
$\GGG$-action on $\zeta$.

\end{enumerate}
The left $\GGG$-action on $F(\GGGG,\zeta)$ with respect to 
right translation in $\GGG$ and with respect to the left
$\GGG$-action on $\zeta$ is
given  by that action which, 
in the group case, corresponds to the diagonal action given
by a formula of the kind \eqref{diag2222}.
Appropriately adjusted to the present situation, 
this formula leads to the following description of the action:
\begin{equation}
(x\cdot \rhoo)(u) = x(\rhoo(ux)),\ x\in \GGG_q, u \in \GGG^q, 
\rhoo \colon \GGG^{\sou(x)} \to E_{\sou(x)},\ q \in \BBB .
\label{2230}
\end{equation}
Indeed, the situation can be depicted by means of 
the commutative diagram
\begin{equation}
\begin{CD}
\GGG^{\sou(x)}
@>{\rho}>>
E_{\sou(x)}
\\
@A{r_x}AA
@VV{\ell_x}V
\\
\GGG^q
@>{x\cdot \rho}>>
E_q 
\end{CD}
\label{CD5}
\end{equation}
where $\ell_x$ refers to left translation in $E$ with $x$
and $r_x$ to right translation in $\GGG$ with $x$.
The formula \eqref{2230}
is exactly the same as (3.5) in \cite{mackeigh}.

\begin{Remark} \label{cd5} The diagram {\rm \eqref{CD5}} 
explains in particular why the 
$\GGG$-action on $F(\GGGG,\zeta)$
cannot be defined by right translation 
in $\GGG$ alone, that is to say, why the action in Example {\rm \ref{ex1}}
and in Example {\rm \ref{ex3}} does {\em not\/} extend to groupoids
since the putative extension of this kind of action
would not be compatible with the variance constraint
imposed by the groupoid axioms. See also the Remark on p.~{\rm 285} of 
{\rm \cite{mackeigh}}.
\end{Remark}

\begin{Remark} \label{topovslie} The construction
of $F(\GGGG,\zeta)$ has been carried out in {\rm \cite{mackeigh}}
for locally trivial locally compact topological groupoids rather than Lie groupoids.
The construction given in
{\rm \cite{mackeigh}} carries over to Lie groupoids as well.
We leave the details to the reader.
\end{Remark}

The previous discussion entails the following observation which
unravels the structure of $F(\GGG,\zeta)$:

\begin{Proposition}\label{unravel1}
Given the left $\GGG$-module $\zeta\colon E \to \BBB$,
for any $q \in \BBB$, the left $\GGG$-module structure on $F(\GGG,\zeta)$
induces a diffeomorphism
\begin{equation}
\GGG^q \times_{\GGG^q_q}C^{\infty}(\GGG^q,E_q) \longrightarrow
F(\GGG,E)
\label{unravel2}
\end{equation}
over $\BBB$.
Thus, suppose that $G$ is a Lie group, that $\xi\colon P \to \BBB$
is a principal right $G$-bundle having $\GGG$ as its gauge groupoid,
and that, as a vector bundle, 
$\zeta$ is the vector bundle associated with $\xi$ and
the left $G$-representation $V$. Then, as a vector bundle,
$F(\GGG,\zeta)$ may be taken to be the vector bundle
\begin{equation}
\xi\times_G C^{\infty}(P,V)\colon P\times_G C^{\infty}(P,V) 
\longrightarrow \BBB
\label{unravel3}
\end{equation}
associated with $\xi$ and the diagonal left $G$-module structure
{\rm \eqref{diag2222}}
on $C^{\infty}(P,V)$ coming from the right $G$-action on $P$ and the
left $G$-action on $V$,
and the left $\GGG$-module structure on $F(\GGG,\zeta)$ is the familiar 
left $\GGG$-module structure on the associated vector bundle
{\rm \eqref{unravel3}}. \qed
\end{Proposition}

The \lq \lq familiar 
left $\GGG$-structure on an associated fiber bundle\rq\rq\ 
is explained, e.~g., in Theorem 1.6.5 (p.~35) of \cite{mackbtwo}.

We will now write the endofunctor $F(\GGGG,\,\cdot \,)$ 
on $\mathrm{Mod}_{\GGG}$ as 
$\mathcal U \colon \mathrm{Mod}_{\GGG} \longrightarrow 
\mathrm{Mod}_{\GGG}$.
Let
\[
\nu\colon \mathcal U^2 \longrightarrow \mathcal U
\]
be the obvious extension of the natural transformation \eqref{nu}.
Thus $\nu$ is the
natural transformation
given by the assignment to
the left $\GGG$-module $\zeta$ of the morphism
\begin{equation*}
\begin{CD}
\nu_{\zeta}
\colon 
F(\GGGG, F(\GGGG,E))
@>>> F(\GGGG,E)
\\
@VVV
@VVV
\\
\BBB
@>{\mathrm{Id}}>>
\BBB
\end{CD}
\end{equation*}
of vector bundles
which, over $q\in \BBB$, sends
$\Psi\colon \GGG^q \to C^{\infty}(\GGG^q,E_q)$
to $\nu_q(\Psi)$ given by
$\nu_q(\Psi)(x) = \Psi(x)(x)$, where $x \in \GGG^q$.

\begin{Remark}\label{groupvsgroupoid2}
The definition of the natural transformation $\nu$ fails to work
with a putative construction involving all the (smooth) maps from 
$\GGGG$ to $E$ over $\BBB$,
$\GGGG$ being viewed over $\BBB$ through $\tar\colon \GGG \to \BBB$,
 rather than $F(\GGGG,E)$ 
since, as pointed out above,
the diagonal map
for groups does {\em not\/} extend to a diagonal map for general groupoids.
Thus, to define cohomology via an appropriate cosimplicial object,
we cannot simply take functions on an associated simplicial object 
of the kind $(E\GGG)^{\mathrm{left}}$
explored in the next section
unless
the groupoid under discussion is a group, cf. Remark 
{\rm \ref{groupvsgroupoid}}. 
\end{Remark}

The system $(\mathcal U,\omega, \nu)$
is a monad over the category ${}_{\GGG}\mathrm{Mod}$.
Given the left $\GGG$-module $\zeta$,
the dual standard construction  $\mathbf U(\zeta)$
associated with  $(\mathcal U,\omega, \nu)$
and the left $\GGG$-module $\zeta$
is a cosimplicial object
in the category of left $\GGG$-modules;
the chain complex $|\mathbf U(\zeta)|$ associated with 
$\mathbf U(\zeta)$,
together with the injection $\omega_{\zeta}$,
is then a differentiably injective resolution of $\zeta$
in the category of left $\GGG$-modules.
The resulting non-normalized chain complex coincides with the {\em standard
resolution\/} (5.4) in \cite{mackeigh},
developed there for locally trivial locally compact topological groupoids.
In view of our categorical point of view
that a monad defines a relative Ext,
that observation establishes the following.

\begin{Theorem}\label{theo1}
Let $\mathcal C$ be the category of continuous vector bundles on $\BBB$
with split morphisms.
Given  
two left $\GGG$-modules $\eta$ and $\zeta$, the chain complex
\[
\mathrm{Hom}_{\GGG}(\eta, |\mathbf U(\zeta)|)
\]
defines 
the relative $\mathrm{Ext}_{(\GGG,\mathcal C)}(\eta,\zeta)$;
in particular, when we take $\eta$ to be
the trivial real line bundle
on $\BBB$ with trivial left $\GGG$-module structure,
this {\rm Ext} comes down to the cohomology of $\GGG$ with values in $\zeta$
introduced in  {\rm \cite{mackeigh}}.
\end{Theorem}

In the special case where $\GGG$ is an ordinary Lie group say $G$
and when $\eta$ is the trivial $G$-representation $\mathbb R$, the functor
$\mathrm{Ext}_{(\GGG,\mathcal C)}(\mathbb R,\, \cdot \,)$
comes down to the cohomology theory developed in \cite{hochmost} and reproduced
in Section \ref{monads} above.

We will now
unravel 
the functor $|\mathbf U(\zeta)|$
in terms of a corresponding principal bundle:
Let $\xi\colon P \to \BBB$ be a  principal right $G$-bundle,
such that the associated
gauge groupoid is isomorphic to $\GGG$.
We can pick any  $q\in \BBB$ and take $G$ to be the vertex group 
$\GGG_q^q$; then 
the projection
$t\colon \GGG^q \to \BBB$ is such a principal right 
$G$-bundle.

Let $\zeta\colon E \to \BBB$ be a left $\GGG$-module.
Relative to $\xi$, as a topological vector bundle,
$\zeta$ amounts to an induced vector bundle of the kind
\[
P\times_G V \longrightarrow \BBB,
\] 
where $V$ is a topological vector space endowed
with a 
left $G$-representation,
so that the standard formalism of associated fiber bundles applies.
Then, relative to $\xi$, as a topological vector bundle,
$F(\GGGG,\zeta)$
amounts to an induced vector bundle of the kind
\[
F(\GGGG,\zeta)\colon P\times_G C^{\infty}(P,V) \longrightarrow \BBB,
\] 
the vector space $C^{\infty}(P,V)$ being suitably topologized
and made into a left $G$-representation via the left 
action \eqref{diag2222},
the requisite left $\GGG$-module structure
on the induced vector bundle $F(\GGGG,\zeta)$ being that
induced from the left $G$-structure \eqref{diag2222}.

Let  $|\mathbf U_{\xi}(V)|$
be the injective resolution of the differentiable left $G$-module
arising from the cosimplicial object
\eqref{UVsmoothxi}. Now,
relative to $\xi$, as a topological differential graded vector bundle,
$|\mathbf U(\zeta)|$
amounts to an induced differential graded vector bundle of the kind
\[
P\times_G |\mathbf U_{\xi}(V)| \longrightarrow \BBB,
\] 
the differential graded left $G$-module
$|\mathbf U_{\xi}(V)|$ being suitably topologized.

The following result, established originally as
Theorem 3 of \cite{mackeigh}
even for  the rigid cohomology of 
a general locally trivial topological groupoid,
is now immediate; we spell it out for later reference,
since we will establish a formally similar result, but
for equivariant cohomology rather than just differentiable
cohomology:

\begin{Proposition}\label{prop9}
Restriction induces an isomorphism
of the cohomology of $\GGG$ with values in $\zeta$
onto the Hochschild-Mostow differentiable cohomology $\mathrm H(G,V)$ 
of $G$ with values in $V$.
\end{Proposition}

\begin{proof} Let $V_1$ and $V_2$ be two left
$G$-representations and
let $\vartheta_1\colon P \times_G V_1 \to \BBB$ and
$\vartheta_2\colon P \times_G V_2 \to \BBB$ be the corresponding left 
$\GGG$-modules.
In view of an observation in \cite{sedaone}, 
$\mathrm{Hom}_{\GGG}(\zeta_1,\zeta_2)$ is naturally isomorphic to
$\mathrm{Hom}_{G}(V_1,V_2)$.
More precisely, the  assignment to a left
$G$-representation $V$ of the $\GGG$-module 
$\vartheta\colon P \times_G V \to \BBB$ 
extends to a functor which is left-adjoint
to the restriction functor
${}_{\GGG}\mathrm{Mod} \to  {}_{G}\mathrm{Mod}$.
Hence, when $\eta$ refers to the trivial line bundle on $\BBB$
with trivial $\GGG$-module structure,
the chain complex $\mathrm{Hom}_{\GGG}(\eta, |\mathbf U_{\xi}(\zeta)|)$
comes down to the chain complex
$\mathrm{Hom}_{G}(\mathbb R,|\mathbf U_{\xi}(V)|)\cong |\mathbf U_{\xi}(V)|^G$ calculating $\mathrm H(G,V)$.
\end{proof}

\section{Comonads and standard constructions} \label{comonads}

Recall that any object $Y$ of a 
symmetric monoidal 
category endowed with 
a cocommutative diagonal---we will take 
the categories of spaces, of
smooth manifolds, of groups, of vector spaces, 
of Lie algebras,
etc.,---defines two simplicial 
objects in the category, the {\em trivial\/}
object which, with an abuse of notation,
we still write as $Y$, and the {\em total object\/} $EY$
(\lq\lq total object\rq\rq\ 
not being standard terminology in this generality); the
trivial object $Y$ has a copy of $Y$ in each 
degree and 
all simplicial
operations are the identity while, for $p \geq 0$, 
the degree $p$  constituent
$EY_p$ of the total object $EY$ is a product of
$p+1$ copies of $Y$ with the familiar 
face operations given by omission and 
degeneracy operations given by insertion.
See e.~g. \cite{bottone} and \cite{koszul}
(1.1).
When $Y$ is an ordinary 
$R$-module,
the simplicial $R$-module
associated with
$Y$ is in fact the result of application of the 
{\em Dold-Kan\/} functor $DK$ 
from chain complexes
to simplicial $R$-modules, 
cf. e.~g. \cite{doldpupp} (3.2 on p.~219).

For illustration, let 
$\fra g$ be a Lie algebra which we 
suppose to be projective over the ground ring $R$.
The total object $E\fra g$ associated with $\fra g$ in the category of 
Lie algebras relative to the obvious monoidal structure
is a simplicial Lie algebra. This construction plays a major role in
the predecessor \cite{koszultw} of the present paper and is 
lurking behind some of the constructions in the next section.

Likewise let $G$ be a group.
When the group $G$ is substituted for $Y$, the resulting simplicial
object is a simplicial group $EG$, and the diagonal injection
$G \to EG$ turns $EG$ into a
simplicial principal right (or left) $G$-set. 
We will refer to $EG$ as the {\em homogeneous\/} universal object
for $G$.

Let $V$ be a right $RG$-module.
The simplicial structure of $EG$ 
and the degreewise right diagonal $RG$-module structures
\eqref{diag22} (with $EG$ substituted for $G$),
relative to the right translation $G$-action on $EG$ where $G$ is viewed
as a subgroup of $EG$ and relative to the $RG$-module structure
on $V$, turn 
\begin{equation}
\mathrm{Map}(EG,V),
\label{EGV}
\end{equation}
into a {\em cosimplicial object in the category of right\/} 
$RG$-modules. Inspection establishes the following.

\begin{Proposition} \label{com1} As a cosimplicial object
in the category of right $RG$-modules, \linebreak
$\mathrm{Map}(EG,V)$ 
coincides with $\mathbf U(V)$ (introduced as {\rm \eqref{U}} above). \qed
\end{Proposition}

Let $(EG)^{\mathrm{left}}$ be the simplicial group having
the iterated semi-direct product
\[
(EG)^{\mathrm{left}}_n
= G \ltimes G \ltimes \ldots \ltimes G\ 
(n+1\ \text{copies of}
\  G) 
\]
as degree $n$ constituent;
the {\em nonhomogeneous\/} face operators 
$\partial_j$ are given by
the familiar formulas
\begin{equation}
\begin{aligned}
\partial_j(x_0,x_1,\dots, x_n) &= 
(x_0,\dots, x_{j-2},x_{j-1}x_j,
x_{j+1},\dots, x_n)\ 
(0 \leq j < n)\\
\partial_n(x_0,x_1,\dots, x_n) &= (x_0,\dots, x_{n-1})
\end{aligned}
\label{face}
\end{equation}
and 
the {\em nonhomogeneous\/} degeneracy operators 
$s_j$ are given by
\begin{equation}
s_j(x_0,x_1,\dots, x_n) = 
(x_0,\dots, x_{j-1},e,x_j,\dots, x_n)
\ (0 \leq j \leq n).
\label{degeneracy}
\end{equation}
The associations
\begin{equation}
(x_0,x_1,\dots, x_n)\longmapsto
(x_0,x_0x_1,\ldots,x_0x_1x_2\ldots x_{n-1},x_0x_1x_2\ldots x_n),\ x_j \in G,
\label{assoc2}
\end{equation}
as $n$ ranges over the natural numbers,
induce an isomorphism of simplicial groups
\begin{equation}
(EG)^{\mathrm{left}} \longrightarrow EG. 
\label{2.5.5.l}
\end{equation}
We will refer to $(EG)^{\mathrm{left}}$ as the {\em nonhomogeneous left
universal object\/} for $G$.

As before, let $V$ be a right $RG$-module. 
We will now consider $\mathrm{Map}((EG)^{\mathrm{left}}, V)$ as a cosimplicial
right $RG$-module, the right $RG$-module structure being the diagonal structure
relative to the right $G$-structure on $V$ and the
left $G$-translation on $(EG)^{\mathrm{left}}$, the cosimplicial
structure being induced from the simplicial structure on 
$(EG)^{\mathrm{left}}$; we recall that the diagonal structure is given by the
association
\begin{equation}
\mathrm{Map}((EG)^{\mathrm{left}}, V) \times G
\longrightarrow
\mathrm{Map}((EG)^{\mathrm{left}}, V),
\ (\alpha,x) \longmapsto \alpha\cdot x,
\label{nonhomogac}
\end{equation}
where $(\alpha\cdot x)y = (\alpha(xy))x$, $x \in G$, 
$y \in (EG)^{\mathrm{left}}$;
here $\alpha$ ranges over maps from $(EG)^{\mathrm{left}}$ to $V$.

\begin{Proposition} \label{prop3}
Relative to the diagonal $G$-action on
$\mathrm{Map}((EG)^{\mathrm{left}}, V)$,
the morphism
\begin{equation}
\Phi=(\varphi_0, \ldots ) 
\colon \mathrm{Map}((EG)^{\mathrm{left}}, V) \longrightarrow \mathbf T(V)
\label{2.5.8}
\end{equation}
of graded $R$-modules which, in degree $n$, is given by the association
\begin{equation}
\begin{aligned}
\varphi_n &\colon\mathrm{Map}(G^{\times(n+1)}, V)
\longrightarrow\mathrm{Map}(G^{\times(n+1)}, \square V),
\\
\varphi_n(\alpha)(x_0,\ldots,x_n) 
&=(\alpha(x_0,\ldots,x_n))\cdot x_0\cdot \ldots \cdot x_n,
\ x_0,\ldots,x_n\in G,
\end{aligned}
\end{equation}
where $\alpha$ ranges over  maps from $G^{\times(n+1)}$ to $V$,
is an isomorphism of cosimplicial right $RG$-modules.
\end{Proposition}

\begin{proof} This comes down to a tedious but straightforward verification.
We leave the details to the reader. 
\end{proof}

We can now give a conceptual explanation for Proposition \ref{prop2}:
The association \eqref{assoc2} induces the isomorphism
\[
\Psi=(\psi_0,\psi_1,\ldots)\colon\mathrm{Map}(EG, V) \longrightarrow
\mathrm{Map}((EG)^{\mathrm{left}}, V)
\]
of cosimplicial right $RG$-modules which, in degree $n \geq 0$,
is given by the association
\begin{equation}
\begin{aligned}
\psi_n\colon \mathrm{Map}(G^{\times (n+1)}, V)
&\longrightarrow 
\mathrm{Map}(G^{\times (n+1)}, \square V)
\\
\psi_n(\alpha)(x_0,x_1,\ldots, x_n) &= 
\alpha (x_0,x_0x_1,\ldots, x_0x_1\ldots x_{n-1}, x_0x_1\ldots x_n)
\end{aligned}
\label{ass3}
\end{equation}
where $x_0,x_1,\ldots x_n \in G$ and 
$\alpha \in \mathrm{Map}(G^{\times (n+1)}, V)$.
The isomorphism $\Psi$,
combined with the isomorphism $\Phi$
in Proposition \ref{prop3}, induces the association spelled out in
Proposition \ref{prop2}.

We will now recall how these constructions can be formalized in the language of
comonads and dual standard constructions.

Let $\mathcal M$ be a category, $\mathcal L\colon
\mathcal M \to \mathcal M$ an endofunctor, 
let $\mathcal I$ denote the identity functor of $\mathcal M$,
and let $\ccc\colon \mathcal L \to  \mathcal I$  
and
$
\delta \colon 
\mathcal L \longrightarrow \mathcal L^2
$
be natural transformations.
Recall that the triple $(\mathcal L,\ccc, \delta)$ is defined 
to be a {\em comonad\/}
over $\mathcal M$ provided the following two diagrams commute:
\begin{equation*}
\begin{CD}
\mathcal L
@>{\delta}>>
\mathcal L^2
@.\phantom{xxxxxxx}@.
\mathcal L
@>=>>
\mathcal L
@>=>>
\mathcal L
\\
@V{\delta}VV
@V{\mathcal L\delta}VV
@.
@V{=}VV
@V{\delta}VV
@V{=}VV
\\
\mathcal L^2
@>{\delta \mathcal L}>>
\mathcal L^3
@.\phantom{xxxxxxx}@.
\mathcal I\mathcal L
@<{\ccc \mathcal L}<<
\mathcal L^2
@>{\mathcal L\ccc}>>
\mathcal L\mathcal I
\end{CD}
\end{equation*}
The natural transformations $\ccc$ and $\delta$ are referred to as the
{\em counit\/} and {\em diagonal\/}, respectively, of the comonad.

The  
{\em  standard construction\/} associated with the comonad 
$(\mathcal L,\ccc, \delta)$,
cf. \cite{duskinon},  \cite{maclbotw},
yields the simplicial object
\[
\left(\mathcal L^{n+1}, d_j\colon
\mathcal L^{n+2} \to \mathcal L^{n+1}, s_j\colon \mathcal L^{n+1} \to \mathcal L^{n+2}
\right)_{n \in \mathbb N};
\]
here, for $n \geq 1$,
\begin{align*}
d_j^n&= \mathcal L^j \ccc 
\mathcal L^{n-j}\colon \mathcal L^{n+1} \to \mathcal L^n,
\ j = 0, \ldots, n,
\\
s_j^n&= \mathcal L^j \delta \mathcal L^{n-j-1}\colon \mathcal L^{n} \to \mathcal L^{n+1},
\ j = 0, \ldots, n-1.
\end{align*}
Thus, given the object $W$ of $\mathcal M$, 
\[
\mathbf L(W)=\left(\mathcal L^{n+1}(W), d_j, s_j\right)_{n \in \mathbb N}
\]
is a simplicial object in $\mathcal M$,
the {\em standard object associated with\/} the object $W$ {\em and the comonad\/}
$(\mathcal L,\ccc, \delta)$,
where we do not distinguish in notation
between the natural transformations  $d_j$ and $s_j$ 
and the morphisms they induce after evaluation of
the corresponding functors at
an object;
under suitable circumstances,
the associated chain complex $\left|\mathbf L(W)\right|$ 
together with $\ccc_W\colon \left|\mathbf L(W)\right|\to W$ 
is then a relatively projective resolution of $W$.

\begin{Example} \label{ex5}
{\rm 
Let $\mathrm{Set}$ be the category of sets 
and ${}_G\mathrm{Set}$
that of left $G$-sets.
The functor $\mathcal V \colon {}_G\mathrm{Set} \to {}_G\mathrm{Set}$
which assigns to the left $G$-set $Z$ the left $G$-set $G \times Z$
endowed with diagonal $G$-action,
together with the natural transformation
$\gamma \colon\mathcal V\to \mathcal V^2 $
which, to the left $G$-set $Z$ assigns the morphism
\[
\gamma\colon G \times Z \to G \times G \times Z,\ 
\gamma(x,q) = (x,x,q),\ x \in G,\ q \in Z,
\]
and the natural transformation
$\alpha\colon \mathcal V\to \mathcal I$ which,
to the left $G$-set $Z$ assigns the morphism
\[
\alpha\colon G \times Z \to Z,\ 
\alpha(x,q) = q,\ x \in G,\ q \in Z,
\]
is a {\em comonad\/} over the category ${}_G\mathrm{Set}$.
The standard construction associated with the comonad 
$(\mathcal V,\alpha,\gamma)$ 
and the $G$-set
consisting of a single point is the simplicial set $EG$, 
that is, the homogeneous universal object associated with $G$,
the simplicial set $EG$ being viewed as a 
principal left $G$-set.
When $G$ is a Lie group, the resulting universal object $EG$ is
a simplicial principal left $G$-manifold.
}
\end{Example}

An adjunction determines a monad
in the following manner  \cite{maclbotw}:
Let $\mathcal F \colon \Cat \to \mathcal M$
be a functor, suppose that the functor $\square\colon \mathcal M \to \Cat$ 
is right-adjoint to $\mathcal F$, and let 
\[
\mathcal L= \mathcal F \square\colon \mathcal M \longrightarrow \mathcal M.
\]
Let $\uuu\colon \mathcal I \to  \square\mathcal F$ be the {\em unit\/},
$\ccc \colon  \mathcal L \to \mathcal I$ the {\em counit\/}
of the adjunction, and let $\delta$ be the natural transformation
\[
\delta = \mathcal F \uuu \square\colon \mathcal L =\mathcal F \square\longrightarrow 
\mathcal F \square\mathcal F \square= \mathcal L^2.
\]
The data $(\mathcal L,\ccc, \delta)$ constitute a {\em comonad\/}
over the category $\mathcal M$.

\begin{Example} \label{ex6}
{\rm 
Let $\mathcal F \colon \mathrm{Set} \to {}_G\mathrm{Set}$
be the functor
which assigns to the
set $Z$ the left $G$-set $G\times Z$, 
endowed with the obvious  
left $G$-action induced by left translation in $G$.
This functor is left adjoint to the forgetful functor 
$\square \colon {}_G\mathrm{Set} \to \mathrm{Set}$,
and the standard construction applied to the 
resulting comonad and the left $G$-set $Z$ yields 
a simplicial set  $\mathcal E(G,Z)$
endowed with a free left
$G$-action. For $Z$ a point $o$, as a simplicial principal left $G$-set,
$\mathcal E(G,o)$ comes down to the nonhomogeneous left universal $G$-object 
$(EG)^{\mathrm{left}}$ considered earlier.
When $G$ is a Lie group, the resulting left universal object 
$(EG)^{\mathrm{left}}$ is a simplicial principal left $G$-manifold.

Here is a comonadic version of Proposition \ref{prop11} above:

\begin{Proposition}
Given the left $G$-set $Z$, the isomorphism of left $G$-sets
\[
\mathcal F \square Z=G\times \square Z \longrightarrow 
G \times Z= \mathcal VZ,\quad (x,q)\mapsto (x,xq),\ x\in G, q \in Z,
\]
induces an isomorphism of comonads 
$(\mathcal L,\ccc, \delta)\longrightarrow (\mathcal V,\alpha, \gamma)$.
\end{Proposition}

This isomorphism of comonads induces 
the isomorphism 
$(EG)^{\mathrm{left}} \to EG$, cf. \eqref{2.5.5.l} above, viewed merely as
an isomorphism
of simplicial principal left $G$-sets. The present
identification in terms of the underlying comonads
yields a conceptual explanation for this isomorphism
of simplicial principal left $G$-sets.
When $G$ is a Lie group, 
that isomorphism is one of simplicial principal left $G$-manifolds.

}
\end{Example}

\begin{Remark} \label{groupvsgroupoid}
{\rm Let $G$ be a Lie group and let $V$ be a differentiable right $G$-module. 
In view of Proposition {\rm \ref{com1}},
the differentiable cosimplicial right $G$-module
$\mathbf U(V)$ introduced as {\rm \eqref{UVsmooth}} above
is simply that of smooth $V$-valued functions on
the simplicial principal left $G$-manifold $EG$.
Likewise, in view of Proposition {\rm \ref{prop3}},
the differentiable cosimplicial right $G$-module
$\mathbf T(V)$ introduced as {\rm \eqref{Tsmooth}} above
is that of smooth $V$-valued functions on
the simplicial principal left $G$-manifold $(EG)^{\mathrm{left}}$.
However, the construction in the previous section shows that
the attempt to introduce cohomology
simply by taking functions
on an appropriate simplicial object
does no longer work for a general groupoid,
cf. Remark 
{\rm \ref{groupvsgroupoid2}}. }
\end{Remark}

In Section \ref{ext} we will introduce yet another example of a comonad
which arises from an extension of Lie-Rinehart algebras.
This example is crucial for the development of the 
equivariant de Rham cohomology relative to a general locally trivial
Lie groupoid.

\section{Lie-Rinehart equivariant cohomology}
\label{lierine}

Let  $\mathfrak g$ be a Lie algebra over a general ground ring $R$.
Recall that,  as a graded Lie algebra, the cone $C\fra g$
in the category of  differential graded $R$-Lie algebras
is the semi-direct product $s\fra g \rtimes \fra g$
of $\fra g$ with the suspension $s\fra g$; here the
suspension $s\fra g$ is just $\fra g$ itself, but regraded up by 1,
and $s\fra g$ is considered as an abelian graded Lie algebra
concentrated in degree 1. The differential $d$ is given by
$d(sY) =Y$, where $Y \in \fra g$.
See \cite{koszultw} for details.

Let $G$ be a Lie group and $\mathfrak g$ its Lie algebra,
and let $C\fra g$ be the cone on $\fra g$ 
in the category of  differential 
graded Lie algebras.
Thus momentarily we are working over the reals as ground ring.
Given a smooth manifold $M$ and a smooth action of $G$ on $M$,
the infinitesimal $\fra g$-action $\fra g\to \mathrm{Vect}(M)$
on $M$ induces a canonical (differential graded) $C\fra g$-module structure
on the de Rham algebra $\mathcal A(M)$ of $M$ via the operations of 
contraction and Lie derivative.
In \cite{koszultw}, we established a close
relationship between the $G$-equivariant de Rham theory
of $M$ and the relative 
$\mathrm{Ext}_{(C\fra g,\fra g)}(\mathbb R,\mathcal A(M) )$
(cf. Example \ref{ex7} above)
where the term relative is used in the sense of \cite{hochsone}.
In particular, when $G$ is compact,
the invariants
$\mathrm{Ext}_{(C\fra g,\fra g)}(\mathbb R,\mathcal A(M))^{\pi_0(G)}$
relative to the group $\pi_0(G)$ of connected components of $G$
coincide with the $G$-equivariant de Rham cohomology of $M$,
cf. \cite{koszultw}.
The question we will explore now is whether
that relative derived $\mathrm{Ext}$ has a meaning for a general
Lie-Rinehart algebra and, if so, what it then signifies.
Let $R$ be an arbitrary commutative ring with 1 and $\mathfrak g$ 
an $R$-Lie algebra, which we suppose to be projective as an $R$-module.
In \cite{koszultw}  
we defined the relative differential
$\mathrm{Ext}_{(C\fra g,\fra g)}$ in terms of a suitable monad;
we have reproduced the details in Example \ref {ex7} above.
Since, over the reals,
when $G$ is a compact and connected Lie group,
$\mathrm{Ext}_{(C\fra g,\fra g)}(\mathbb R,\mathcal A(M))$
coincides with the $G$-equivariant de Rham cohomology of $M$,
we refer to 
$\mathrm{Ext}_{(C\fra g,\fra g)}(R,\,\cdot \,)$
as the {\em infinitesimal equivariant cohomology\/}
relative to $\fra g$.
In \cite{koszultw}, we also spelled out a suitable
comonad which, in turn, leads to the corresponding 
relative bar resolution, where the term \lq\lq relative\rq\rq\ 
is intended to hint at the
fact that this is the bar resolution for the relative situation
under discussion, cf. e.~g. \cite{maclaboo} (Ch. 9 and Ch. 10).
Since $\mathrm U[C\fra g]$ is actually a Hopf algebra,
a {\it homogeneous\/} version of the corresponding relative
bar resolution is available as well;
this  homogeneous
relative bar resolution
leads to the {\it simplicial Weil coalgebra\/} explored in
\cite{koszultw}. The construction of the simplicial 
Weil coalgebra does not extend to a general Lie-Rinehart algebra.
However, the construction of the dual object, the CCE 
algebra or, equivalently, {\em Maurer-Cartan algebra\/}, cf. \cite{vanesthr}
for this terminology, 
extends to a general Lie-Rinehart algebra.
This leads to a notion of {\em Lie-Rinehart equivariant cohomology\/}.
We will now explain the details.

We begin with some preparations.
Let $\mathfrak g$ and $\mathfrak g'$ be ordinary Lie algebras
and let $\Phi\colon \mathfrak g \to \mathfrak g'$ be a morphism
of Lie algebras. Let $N$ be a left $\mathfrak g$-module, $N'$ 
a left $\mathfrak g'$-module,
view 
$N'$ 
as a left $\mathfrak g$-module via $\Phi$,
and let $\varphi \colon N'\to N$ be a morphism of $R$-modules.
Recall that the pair $(\varphi,\Phi)$ is referred to as a {\em morphism of actions\/}
$(N,\mathfrak g) \to (N',\mathfrak g')$ provided $\varphi$ is a morphism of 
$\mathfrak g$-modules.
We will use the notation
$\mathrm{Alt}(\mathfrak g,N)$ etc. for the ordinary graded
$R$-module or chain complex of $N$-valued alternating forms
on $\mathfrak g$.
It is a classical fact (and entirely obvious) that 
a {\em morphism of actions\/}
$(N,\mathfrak g) \to (N',\mathfrak g')$ induces the morphism
\begin{equation}
\begin{CD}
\mathrm{Alt}(\mathfrak g',N')
@>{\Phi^*}>>
\mathrm{Alt}(\mathfrak g,N')
@>{\varphi_*}>> \mathrm{Alt}(\mathfrak g,N)
\end{CD}
\label{ce1}
\end{equation}
between the CCE complexes and hence the morphism
\[
\mathrm H^*(\mathfrak g',N')
\longrightarrow\mathrm H^*(\mathfrak g,N)
\]
on cohomology.
The morphism \eqref{ce1} can also be written as the composite
\begin{equation}
\begin{CD}
\mathrm{Alt}(\mathfrak g',N')
@>{\varphi_*}>>
\mathrm{Alt}(\mathfrak g',N)
@>{\Phi^*}>> \mathrm{Alt}(\mathfrak g,N)
\end{CD}
\label{ce2}
\end{equation}
of morphisms of graded $R$-modules
but the graded $R$-module $\mathrm{Alt}(\mathfrak g',N)$
does not in general acquire a differential compatible with the
other structure since only $\mathfrak g$ is supposed to act on $N$,
not $\mathfrak g'$.

Adapting the notion of morphism of actions to Lie-Rinehart algebras
via the description \eqref{ce2}
leads to the notion of {\em comorphism\/} of Lie-Rinehart algebras,
introduced in \cite{higmathr}.
A comorphism is not the dual of a morphism, though.
The corresponding notion for
Lie algebroids is that of {\em morphism of Lie algebroids\/} and goes back to
\cite{almekump}.
Thus, let $(A,L)$ and $(A',L')$ be Lie-Rinehart algebras.
A {\em comorphism\/} $(\varphi,\Phi)\colon (A,L) \to (A',L')$ of Lie-Rinehart algebras
consists of a morphism $\varphi\colon A' \to A$ of algebras and a morphism
$\Phi \colon L \to A\otimes_{A'}L'$ of $A$-modules such that (i) 
and (ii) below 
are satisfied:

\noindent
(i) The diagram
\begin{equation}
\begin{CD}
L\otimes A'
@>{L \otimes \varphi}>>L\otimes A
@>{\mathrm{action}}>> A
\\
@V{\Phi \otimes A'}VV
@.
@VV{A}V
\\
A\otimes_{A'}L'\otimes A'
@>>{A\otimes \mathrm{action}}> A\otimes_{A'} A'
@>>{\mathrm{can}}>
A
\end{CD}
\label{i}
\end{equation}
is commutative.

\noindent
(ii) Given $\aalpha_1,\aalpha_2 \in L$ with $\Phi(\aalpha_1) =\sum a^1_{i} \beta^1_{i}$ and
$\Phi(\aalpha_2) =\sum a^2_{j} \beta^2_{j}$,
\begin{equation}
\Phi([\aalpha_1,\aalpha_2])=
\sum a^1_{i} a^2_{j} \otimes[\beta^1_{i},\beta^2_{j}]
+
\sum \aalpha_1(a^2_{j})\otimes \beta^2_{j} 
-
\sum \aalpha_2(a^1_{i})\otimes \beta^1_{i} .
\label{ii}
\end{equation}

It is not required that $\Phi$ lift to a morphism
$L \to L'$ of $R$-modules or even $R$-Lie algebras.

\smallskip\noindent
{\sc Example\/.} 
A smooth map $f\colon M \to N$ between smooth manifolds induces the comorphism
\[
(f^*,f_*)\colon (C^{\infty}(M),\mathrm{Vect}(M))\longrightarrow 
(C^{\infty}(N),\mathrm{Vect}(N))
\]
of Lie-Rinehart algebras in an obvious manner;
in particular
$f^*\colon C^{\infty}(N)\to C^{\infty}(M)$
is the induced morphism of commutative algebras.

A {\em comorphism\/} $(\varphi,\Phi)\colon (A,L) \to (A',L')$ of Lie-Rinehart algebras
induces the morphism
\begin{equation}
\begin{CD}
\mathrm{Alt}_{A'}(L',A')
@>{\varphi_*}>>
\mathrm{Alt}_{A}(A\otimes _{A'}L',A)
@>{\Phi^*}>>
\mathrm{Alt}_{A}(L,A)
\end{CD}
\label{MCA}
\end{equation}
of Maurer-Cartan algebras. As a morphism of graded $A'$-algebras,
this morphism is well defined in view of (i). The requirement (ii)
entails the compatibility with the differentials.

Let $(A,L)$ be a Lie-Rinehart algebra.
The cone $CL$
in the category of differential graded
$R$-Lie algebras has no obvious meaning
as a differential graded Lie-Rinehart algebra. Indeed, this cone
contains the suspended object $sL$ as a graded Lie ideal
but {\em not\/} as a differential graded Lie ideal,
and the injection of $L$ into $CL$ does not induce
any kind of $CL$-action on $A$ whatsoever.

On the other hand, the Lie-Rinehart axioms imply that the 
(differential graded) action
\begin{equation}
CL \otimes \mathrm{Alt}_R(L,\MMM) \longrightarrow \mathrm{Alt}_R(L,\MMM)
\end{equation}
passes to an action
\begin{equation}
CL \otimes \mathrm{Alt}_A(L,\MMM) \longrightarrow \mathrm{Alt}_A(L,\MMM)
\end{equation}
of $CL$ on
$\mathrm{Alt}_A(L,\MMM)$
which is in fact a differential graded
$(CL)$-action on $\mathrm{Alt}_A(L,\MMM)$.
Since the operation of contraction is compatible with the $A$-module structure
in the sense that 
$i_{a\aalpha} = ai_{\aalpha}$ for $a \in A$ and $\aalpha \in L$,
the operation
$\lambda$ of Lie derivative necessarily satisfies the familiar 
identity
\begin{equation}
\lambda_{a\aalpha}(\omega)= a\lambda_{\aalpha}(\omega) + 
da\cup i_{\aalpha}(\omega)\ (a \in A, \ \aalpha \in L,\ 
\omega \in \mathrm{Alt}_A(L,\MMM)).
\label{CL}
\end{equation}
In particular, relative to the operation of Lie-derivative,
$\mathrm{Alt}_A(L,\MMM)$ is just an $L$-module,
not an $(A,L)$-module.

Let $\mathcal N$ be a (differential graded) 
$\mathrm{Alt}_A(L,A)$-module which is, 
furthermore, endowed with a $(CL)$-module structure;
for intelligibility we will write the $(CL)$-module structure 
on $\mathcal N$ 
in terms of the familiar notation $\lambda$ of Lie-derivative and
$i$ of contraction.
Abstracting from the property
just isolated, we will refer to $\mathcal N$ as an 
$(A,CL)$-{\em module\/} whenever the identity \eqref{CL}
is satisfied, with $\mathcal N$ substituted for
$\mathrm{Alt}_A(L,\MMM)$.

Consider the differential graded 
algebra $\mathrm U_A[CL]$ of operators on
$\mathrm{Alt}_A(L,A)$ generated by $A$ and $CL$.
In this algebra, given
$a \in A$ and $\aalpha \in L$, the operators 
$\lambda_{a\aalpha}$, $\lambda_{\aalpha}$ and $i_\aalpha$ are related by the identity
\begin{equation}
\lambda_{a\aalpha}-a\lambda_{\aalpha}= (da)i_\aalpha.
\label{id1}
\end{equation}
In order for this identity to make sense we must extend the
coefficients from $A$ to
$\mathrm{Alt}_A(L,A)$.
Thus, consider the crossed product algebra
\[
\mathrm U[L]\odot \mathrm{Alt}_A(L, \Lambda_A[sL]);
\]
requiring that the  identity \eqref{id1} be satisfied
leads us to the  quotient algebra
such that this identity holds.
Consequenctly the differential graded 
algebra $\mathrm U_A[CL]$ of operators on
$\mathrm{Alt}_A(L,A)$ generated by $A$ and $CL$
is a certain quotient of that algebra.

For an ordinary Lie algebra $\fra g$,
in \cite{koszultw}, we introduced the {\em simplicial Weil coalgebra\/}
of $\fra g$ to be the simplicial CCE coalgebra
$\Lambda'_{\partial}[sE\fra g]$ of the total simplicial object
$E\fra g$ associated with $\fra g$. For reasons explained in
\cite{bv}, the construction of the CCE coalgebra does not 
extend to a general Lie-Rinehart algebra. However the dual object,
the {\em CCE algebra\/} or {\em Maurer-Cartan\/} algebra, extends; in fact, 
this is just the differential graded commutative algebra
$\mathrm{Alt}_A(L,A)$ associated with the Lie-Rinehart algebra
$(A,L)$. 
On the symmetric monoidal
category of differential graded cocommutative coalgebras over $R$,
the {\em total\/} object functor associates with any
differential graded cocommutative coalgebra
$C$ the simplicial differential graded coalgebra $EC$
whose condensed object $|EC|^{\bullet}$ is a contractible
differential graded coalgebra in the sense that
the counit $|EC|^{\bullet} \to R$
is a chain equivalence.

On the symmetric monoidal
category of differential graded commutative algebras over $R$,
the construction of the total simplicial object dualizes
to that of a {\em total cosimplicial\/} differential graded algebra
functor $\mathcal C$;
this functor assigns to
the differential graded commutative algebra $\mathcal A$
the differential graded commutative algebra
$\mathcal C \mathcal A$
having, for $n \geq 0$, the tensor power  
$(\mathcal C \mathcal A)_{n}= \mathcal A^{\otimes (n+1)}$
 {\em over the ground ring\/} $R$ of
$n+1$ copies of $\mathcal A$
as degree $n$ constituent, 
with coface and codegeneracy operators being induced by the
multiplication map 
$\mu \colon \mathcal A\otimes \mathcal A \to \mathcal A$
of $\mathcal A$ and the unit map $\uuu \colon R \to A$ of $A$.
Totalization and normalization yields the 
differential 
graded commutative algebra
\begin{equation}
\left|\mathcal C \mathcal A\right|^{\bullet}
\end{equation}
which is a contractible
differential graded algebra in the sense that
the unit $R \to |\mathcal C \mathcal A|^{\bullet}$
is a chain equivalence.

This construction applies, in particular, to 
the differential graded $R$-algebra
$\mathcal A=\mathrm{Alt}_A(L,A)$ but, beware,
the tensor powers are taken over the ground ring $R$
(and not over $A$; indeed the construction
would not then be well defined). Furthermore, 
$CL$ being an ordinary differential graded $R$-Lie algebra,
the $(CL)$-action on $\mathrm{Alt}_A(L,A)$ extends to a
$(CL)$-action on $\mathcal C\mathrm{Alt}_A(L,A)$
and hence on
$\left|\mathcal C \mathrm{Alt}_A(L,A)\right|^{\bullet}$, and the invariants
$\left(\left|\mathcal C \mathrm{Alt}_A(L,A)\right|^{\bullet}\right)^{CL}$ 
constitute a differential graded
$R$-subalgebra.

For illustration, consider the special case where 
$A$ coincides with the ground ring and where
$L$
is an ordinary Lie algebra $\fra g$.
Then 
$\left(\left|\mathcal C \mathrm{Alt}_A(L,A)\right|^{\bullet}\right)^{CL}$ 
comes down to 
$\left(\left|\mathcal C \mathrm{Alt}(\fra g,R)\right|^{\bullet}\right)^{C\fra g}$, and
$\mathcal C \mathrm{Alt}(\fra g,R)$ is canonically isomorphic to
\[
\mathrm{Hom}(\Lambda'_{\partial}[sE\fra g],R) \cong 
\mathrm{Hom}(E\Lambda'_{\partial}[s\fra g],R)\]
 whence
\begin{equation}
\mathrm H\left(\left(\left|\mathcal C \mathrm{Alt}
(\fra g,R)\right|^{\bullet}\right)^{C\fra g}\right)
\cong \mathrm{Ext}_{(C\fra g, \fra g)}(R,R) .
\end{equation}
In particular, when $\fra g$ is reductive,  $\mathrm H^*(\fra g)$
is the exterior Hopf algebra $\Lambda[\mathrm{Prim}(\fra g)]$
on its primitives $\mathrm{Prim}(\fra g)\subseteq \mathrm H^*(\fra g)$, and
\begin{equation}
\mathrm{Ext}_{(C\fra g, \fra g)}(R,R)
\cong
\mathrm S[s^{-1}\mathrm{Prim}(\fra g)],
\end{equation}
the symmetric algebra
on the desuspension  $s^{-1}\mathrm{Prim}(\fra g)$
of the
primitives $\mathrm{Prim}(\fra g)\subseteq \mathrm H^*(\fra g)$.

Slightly more generally, suppose that $\fra g$ acts on the 
commutative algebra $A$ by derivations, and let
$L= A \odot \fra g$ be the resulting crossed product
$(R,A)$-Lie algebra, cf. e.~g. \cite{poiscoho}. Then the differential graded algebra
$\mathrm{Alt}_A(L,A)$ comes down to
$\mathrm{Alt}(\fra g,A)$
 and
$\mathcal C \mathrm{Alt}_A(L,A)$ is canonically isomorphic to
$\mathcal C \mathrm{Alt}(\fra g,A)$
or
\[
\mathrm{Hom}(\Lambda'_{\partial}[sE\fra g],A) \cong 
\mathrm{Hom}(E\Lambda'_{\partial}[s\fra g],A).
\]
When $\fra g$ is reductive,
$\mathrm H^*(\fra g,A) \cong \Lambda[\mathrm{Prim}(\fra g)] \otimes A^{\fra g}$, and
\begin{equation}
\mathrm H\left(\left(\left|\mathcal C \mathrm{Alt}
(\fra g,A)\right|^{\bullet}\right)^{C\fra g}\right)
\cong  \mathrm S[s^{-1}\mathrm{Prim}(\fra g)] \otimes A^{\fra g}.
\end{equation}

To unravel the structure, we note that, 
relative to the obvious structures,
$(A\otimes A,L\otimes A)$ and
$(A\otimes A,A\otimes L)$ acquire Lie-Rinehart 
structures, and the direct sum
\begin{equation}
L^{\times}=L\otimes A \oplus A\otimes L
\label{dir}
\end{equation} 
of $(A\otimes A)$-modules
acquires an obvious $(R,A \otimes A)$-Lie algebra structure.
We use the notation $L^{\times}$ since 
$(A\otimes A,L^{\times})$ is the categorical product 
of $(A,L)$ with itself in the category of Lie-Rinehart
algebras, and we will therefore occasionally write
\begin{equation}
(A,L)\times (A,L)= (A\otimes A,L^{\times}).
\label{prod}
\end{equation}
The corresponding product construction for Lie algebroids
was introduced in \cite{almekump}.
A special case of the construction in \cite{almekump} is this:
Over the reals $\mathbb R$ as ground ring,
when $L$ is the $(\mathbb R,C^{\infty}(\MMM))$-Lie algebra
of smooth vector fields on the smooth manifold $\MMM$,
when the ordinary tensor product is replaced with a
suitably completed tensor product (Frech\'et tensor product),
\eqref{dir} is the familiar 
decomposition of the Lie algebra of smooth vector 
fields on the product $\MMM \times \MMM$
into two summands corresponding 
to the two factors $\MMM$.
Accordingly
\begin{equation}
\mathrm{Alt}_{A \otimes A} (L^{\times},A\otimes A)
\cong
\mathrm{Alt}_{A}(L,A) \otimes \mathrm{Alt}_{A}(L,A)
\end{equation} 
canonically.
The ordinary diagonal morphism
\[
\Delta \colon L \longrightarrow 
L \oplus L \cong A \otimes_{A\otimes A}L^{\times},
\ 
\Delta(\rhoo) =  (\rhoo,\rhoo)
\ (\rhoo \in L)
\]
of $A$-modules
and the multiplication map $\mu \colon A \otimes A \to A$
constitute a comorphism
\begin{equation}
(\mu,\Delta)\colon (A,L) \longrightarrow (A \otimes A,  L^{\times})
=(A,L)\times (A,L)
\end{equation}
of Lie-Rinehart algebras,
the {\em  Lie-Rinehart diagonal map\/} for $(A,L)$.
In this case,
the defining condition \eqref{i} of a comorphism
simply comes down to the fact that
$L$ acts on $A$ by derivations and the defining condition \eqref{ii} is 
plainly satisfied.

The induced morphism of Maurer-Cartan algebras
of the kind \eqref{MCA}
coincides with the multiplication map of $\mathrm{Alt}_A(L,A)$.

With the notion of comorphism as morphism, Lie-Rinehart algebras
constitute a symmetric monoidal category, and the 
Lie-Rinehart diagonal is a diagonal morphism in this category.
Consequently, 
given the Lie-Rinehart algebra $(A,L)$, the {\em total\/} object
\begin{equation}
E(A,L) = (EA,EL)
\label{tot1}
\end{equation}
is defined as a {\em simplicial\/} Lie-Rinehart algebra.
Thus $EL$ is a simplicial $R$-Lie algebra,
$EA$ is a cosimplicial $R$-algebra, and the structure is encoded in
two pairings
\begin{equation}
EA \otimes EL \longrightarrow EL,
\
EL \otimes EA \longrightarrow EA
\label{tot2}
\end{equation}
satisfying the appropriate Lie-Rinehart axioms.

Application of the Maurer-Cartan functor
$\mathrm{Alt}_{(\,\cdot \,)}(\,\cdot \,,\,\cdot \, )$
to the simplicial Lie-Rinehart algebra $(EA,EL)$
yields the cosimplicial differential graded algebra
\begin{equation}
\mathrm{Alt}_{EA}(EL,EA).
\label{smc2}
\end{equation}
However this is precisely the
cosimplicial differential graded algebra
$\mathcal C\mathrm{Alt}_{A}(L,A)$ considered earlier, 
and we will refer to it as
the {\em cosimplicial Maurer-Cartan algebra associated with
the simplicial Lie-Rinehart algebra\/} $(EA,EL)$.
Furthermore, the $(CL)$-action on 
$\mathrm{Alt}_{A}(L,A)$ by contraction and Lie-derivative
extends canonically to a $(CL)$-action on 
$\mathrm{Alt}_{EA}(EL,EA)$ compatible with the cosimplicial 
structure in the sense that each cosimplicial operator
is compatible with the $(CL)$-actions.

Let $\mathcal N$ be an $(A,CL)$-module
and view it as the trivial cosimplicial $(A,CL)$-module.
The product $(EA)\otimes\mathcal N$
of cosimplicial objects is defined
and inherits a canonical
$(A,CL)$-module structure
compatible with the cosimplicial structure.
Application of the Maurer-Cartan functor
$\mathrm{Alt}_{(\,\cdot \,)}(\,\cdot \,,\,\cdot \, )$
to the simplicial Lie-Rinehart algebra $(EA,EL)$
yields the cosimplicial differential graded
$\mathrm{Alt}_{EA}(EL,EA)$-module
\begin{equation}
\mathrm{Alt}_{EA}(EL,(EA)\otimes\mathcal N).
\label{smc3}
\end{equation}
We define the $(A,L)$-equivariant cohomology
$\mathrm H^*_{(A,L)}(\mathcal N)$ of $\mathcal N$
 to be the homology
\begin{equation}
\mathrm H^*_{(A,L)}(\mathcal N)=
\mathrm H\left(
\mathrm{Alt}_{EA}(EL,(EA)\otimes\mathcal N)^{CL}
\right).
\label{eq}
\end{equation}
We can view this theory also as the relative
\begin{equation}
\mathrm{Ext}_{(A,CL,L)}(EA,\mathcal N).
\label{extr}
\end{equation}

\noindent
{\sc Illustration\/.}
Let $\mathcal F$ be a foliation of the smooth manifold
$M$, let $A=C^{\infty}(M)$,
let $\tau_{\mathcal F}$ be the tangend bundle of $\mathcal F$, 
and let $L$ be the $(\mathbb R,A)$-Lie algebra of smooth vector
fields tangent to $\mathcal F$. Thus the injection
$L \to L_M=\mathrm{Vect}(M)$ is a morphism of $\mathbb R,A)$-Lie 
algebras, and the de Rham complex 
$\mathcal N=\mathcal A(M)=\mathrm{Alt}_A(L_M,A)$
of $M$ is an $(A,L,CL)$-module via the $A$-module structure
and the operations of contraction and Lie derivative.
Under these circumstances,
$\mathrm H^*_{(A,L)}(\mathcal N)$ is the
{\em de Rham cohomology of $M$ that is equivariant relative to the
foliation\/} $\mathcal F$.

\section{Extensions of Lie-Rinehart algebras}
\label{ext}
The algebraic analog of an \lq\lq Atiyah sequence\rq\rq\ or of a
\lq\lq transitive Lie algebroid\rq\rq\ (see below for details) is
an extension of Lie-Rinehart algebras \cite{extensta}.

Let $L'$, $L$, $L''$ be $(R,A)$-Lie algebras. An {\em extension\/}
of $(R,A)$-Lie algebras is a short exact sequence
\begin{equation}
\begin{CD}
\mathrm {\mathbf e} \colon 0 @>>> L' @>>> L @>p>> L''
\longrightarrow 0 \end{CD}
\label{2.1}
\end{equation}
in the category of $(R,A)$-Lie algebras; notice in particular that
the Lie algebra $L'$ necessarily acts trivially on $A$.

\begin{Example} \label{2.222}
{\rm Let the ground ring $R$ be the field $\mathbb R$ of real numbers,
let $B$ be a smooth finite dimensional manifold, let $A$ be the
algebra of smooth functions on $B$, let $G$ be a Lie group, and
let $\xi \colon \QQ \to B$ be a principal right $G$-bundle. The vertical
subbundle $\mathrm T \QQ_{\mathrm{vert}} \to \QQ$ of the tangent bundle 
$\tau_\QQ\colon \mathrm T\QQ \to B$
of $\QQ$ is well known to be trivial (beware, not equivariantly
trivial), having as fibre the Lie algebra $\fra g$ of $G$, that
is, $\mathrm T \QQ_{\mathrm{vert}} \cong \QQ \times \fra g$. 
Dividing out the $G$-actions, we
obtain an extension
\begin{equation}
0 \longrightarrow \mathrm{ad}(\xi) \longrightarrow \tau_\QQ/G
\longrightarrow \tau_B \longrightarrow 0 \label{2.2.1}
\end{equation}
of vector bundles over $B$, where $\tau_B$ is the tangent bundle
of $B$. This sequence has been introduced by {\rm
Atiyah~\cite{atiyaone}} and is now usually referred to as the {\em
Atiyah sequence \/} of the principal bundle $\xi$; here
$\mathrm{ad}(\xi)\colon P\times_G\fra g \longrightarrow B$ 
is the bundle associated with the principal
bundle and the adjoint representation of $G$ on its Lie algebra
$\fra g$. The spaces ${\fra g}(\xi)= \Gamma\mathrm{ad}(\xi)$ and
$\mathrm E(\xi)=\Gamma(\tau_\QQ/G)\cong \Gamma(\tau_\QQ)^G$ 
of smooth sections inherit
obvious Lie algebra structures, in fact $(\mathbb R,A)$-Lie
algebra structures, and
\begin{equation}
0 \longrightarrow \fra g(\xi) \longrightarrow \mathrm E(\xi)
\longrightarrow \mathrm{Vect}(B) \longrightarrow 0 \label{2.2.2}
\end{equation}
is an extension of  $(\mathbb R,A)$-Lie algebras; here
$\mathrm{Vect}(B)$ is the $(\mathbb R,A)$-Lie algebra of smooth vector
fields on $B$, and $\fra g(\xi)$ is well known to be, in an obvious way,
the Lie
algebra of the {\em group\/} of {\em gauge transformations\/}
 of
$\xi$, see Section \ref{algebroids} below. The resulting Lie algebroid
\begin{equation}
\tau_\QQ\big/G\colon (\mathrm T\QQ)\big/G \longrightarrow B
\label{atiy2}
\end{equation}
over $B$ is plainly transitive in the sense that
the \lq\lq anchor\rq\rq\ $\mathrm E(\xi)\to \mathrm{Vect}(B)$
is surjective.
We will refer to this Lie algebroid as the {\em Atiyah algebroid\/}
of the principal bundle $\xi$.
}
\end{Example}

Over a general ground ring $R$, consider an extension 
\begin{equation*}
\mathrm {\mathbf e} \colon 0 \longrightarrow L' \longrightarrow L 
\longrightarrow L''
\longrightarrow 0 
\end{equation*}
of $(R,A)$-Lie algebras
of the kind \eqref{2.1}.
To avoid unnecessary complications,
we will suppose that $L'$, $L$, and $L''$
are finitely generated and projective as $A$-modules.
Since $L'$ is an ordinary $A$-Lie algebra
that acts trivially on $A$,
unlike the situation in Section \ref{lierine} above,
the cone $C^{\mathrm {\mathbf e}}L = (sL')\rtimes L$
in the category of differential graded
$(R,A)$-Lie algebras makes perfect sense
as the differential graded Lie-Rinehart algebra
$(A,C^{\mathrm {\mathbf e}}L )$;
more precisely, the projection from $L$ to $L''$ extends to a morphism
$C^{\mathrm {\mathbf e}}L \to L''$
of differential graded Lie algebras and, in this way,
the pair $(A,C^{\mathrm {\mathbf e}}L )$ acquires
a differential graded Lie-Rinehart algebra structure
in such a way that
\[
(A,C^{\mathrm {\mathbf e}}L )\longrightarrow (A,L'' )
\]
is a morphism of differential graded Lie-Rinehart algebras.
We note that, by construction, the cone $C^{\mathrm {\mathbf e}}L$
contains the suspended object $sL'$ as a graded Lie ideal
but {\em not\/} as a differential graded Lie ideal since
the values of the differential, applied to elements of
$sL'$, lie in $L$.  
The notation $C^{\mathrm {\mathbf e}}$
is intended to indicate a cone construction, whence
the letter $C$, but relative the the extension $\mathbf e$,
whence the superscript.

The notion of $(A,C^{\mathrm {\mathbf e}}L)$-module is defined
in the obvious way and, playing formally the same game
as in Section \ref{lierine} above,
we introduce the
relative
\begin{equation}
\mathrm{Ext}_{(A,C^{\mathrm {\mathbf e}}L,L)}
\label{ext1}
\end{equation}
via the appropriate comonad as follows:

Let $\mathrm U_A[L]$ be the universal algebra associated with the
Lie-Rinehart algebra $(A,L)$ and
$\mathrm U_A[C^{\mathbf e} L]$
the universal differential graded algebra associated with the
differential graded Lie-Rinehart algebra $(A,C^{\mathbf e} L)$.
Consider the pair
$(\mathcal R,\mathcal S)=(\mathrm U_A[C^{\mathbf e} L], \mathrm U_A[L])$,
with $\mathcal R \supseteq \mathcal S$.
This pair gives rise to a  resolvent pair of 
categories, cf.
\cite{maclaboo} (IX.6), that is to say, the functor
\[
\mathcal F \colon \mathrm{Mod}_{\mathcal S} \longrightarrow  \mathrm{Mod}_{\mathcal R}
\] 
which assigns to the
right $\mathcal S$-module $N$ the induced  $\mathcal R$-module 
$\mathcal F(N)=N\otimes_{\mathcal R}\mathcal S$
is left adjoint to the forgetful functor $\square 
 \colon \mathrm{Mod}_{\mathcal R} \to  \mathrm{Mod}_{\mathcal S}$.
Relative (co)homology is then defined and
calculated in terms of a relatively projective 
resolution in the sense of  \cite{hochsone}.
Given the right $\mathcal R$-module $N$, the
{\em standard construction\/} $\mathbf L(N)$
arising from $N$ and the {\em comonad\/}
$(\mathcal L,\ccc,\delta)$
associated with the adjunction is a simplicial object
whose associated chain complex
$\big|\mathbf L(N)\big|$
coincides with the
standard relatively projective resolution of $N$ 
in the sense of  \cite{hochsone}.

Let  $\chain_{(A,L)}$ be the category of
chain complexes in the category of right $(A,L)$-modules and let
\[
\mathcal F \colon \chain_{(A,L)} \longrightarrow \mathrm{Mod}_{(A,C^{\mathbf e}L)}
\] 
be the functor which
assigns to the right $(A,L)$-chain complex $N$ the
totalized CCE complex 
\[
\mathcal FN = N \otimes_{(A,L)} \mathrm U_A[C^{\mathbf e} L] \cong
N \otimes_{(A,\tau_{L'})}(\Lambda_A')_{\partial}[sL']
\]
calculating the Lie algebra homology of the $A$-Lie algebra
$L'$ with coefficients
in $N$, viewed as an $(A,L')$-module
 (suitably interpreted relative to the chain complex structure on $N$).
Here $(\Lambda_A')_{\partial}[sL']$ refers to
CCE coalgebra of the $A$-Lie algebra $L'$, taken in the category of
$A$-modules, and the notation
$\,\cdot \, \otimes_{(A,\tau_{L'})}\,\cdot\,$
refers to the twisted tensor product relative to the universal twisting cochain
$\tau_{L'}\colon(\Lambda_A')_{\partial}[sL']\longrightarrow 
\mathrm U_A[L']$ in the category of $A$-modules.
Now, on the category of $(A,C^{\mathbf e}L)$-modules, the differential
functors
$\mathrm{Tor}^{(A,C^{\mathbf e} L, L)}$
and
$\mathrm{Ext}_{(A,C^{\mathbf e} L, L)}$ are defined:
Given the 
right  $(A,C^{\mathbf e}L)$-module $\mathbf W$ and the left
 $(A,C^{\mathbf e}L)$-module $\mathbf U$, the {\em relative differential 
$\mathrm{Tor}^{(A,C^{\mathbf e} L, L)}(\mathbf W,\mathbf U)$\/} 
is the homology of
the chain complex
\[
\big|\mathbf L(\mathbf W)\big|\otimes_{(A,C^{\mathbf e}L)} \mathbf U.
\]
Given the 
right  $(A,C^{\mathbf e}L)$-modules 
$\mathbf W$ and $\VV$, the {\em relative differential 
$\mathrm{Ext}_{(A,C^{\mathbf e} L, L)}(\mathbf W,\VV)$\/} is the homology of
the chain complex
\[
\mathrm{Hom}_{(A,C^{\mathbf e}L)}\left(
\big|\mathbf L(\mathbf W)\big|,\VV\right).
\]

\section{The cone on the Lie algebroid associated with a Lie groupoid}
\label{algebroids}

Let $\GGG$ be a Lie groupoid and
let $\tau_\GGG^{\tar}\colon \mathrm T^{\tar} \GGG \to \GGG$ 
be the tangent bundle of the foliation given by the $\tar$-fibers.
Recall that, as a smooth vector bundle,
the {\em Lie algebroid\/} $\lambda_\GGG\colon A\GGG \to \BBB$ {\em associated with\/} $\GGG$
is the pull back vector bundle $1^*\tau_\GGG^{\tar}$ across the
object inclusion map $1\colon \BBB \to \GGG$ of
$\tau_\GGG^{\tar}$.
Since the left-invariant vector fields on $\GGG$ are closed
under the Lie bracket, the space
$\Gamma(\lambda_\GGG)$ of sections is well known
to acquire a Lie bracket, the 
{\em Lie algebroid bracket\/}.
Furthermore, the {\em anchor\/}
$a_\GGG \colon  \lambda_\GGG \to \tau_{\BBB}$ is the vector bundle morphism
which maps each section $X$ of $\lambda_\GGG$ to the $\sou$-projection
of the corresponding left-invariant vector field.
The Lie algebroid associated with $\GGG$ is also the Atiyah algebroid
of an associated principal bundle
$\xi\colon P \to \BBB$
having $\GGG$ as its
gauge groupoid.
We will write the resulting Lie-Rinehart algebra as
\begin{equation}
\left(A_{\BBB}, L_\GGG\right)=\left(C^{\infty}(\BBB), \Gamma(\lambda_\GGG)\right) .
\label{lr1}
\end{equation}
The anchor $a_\GGG$  amounts to a morphism
of Lie-Rinehart algebras
from $\left(A_{\BBB}, L_\GGG\right)$ to
$\left(A_{\BBB}, \mathrm{Vect}(\BBB)\right)$.

Let  $\theta\colon E \to \BBB$ be a  
vector bundle. Recall that
a left $\GGG$-module structure
on $\theta$ induces a left
$\left(A_{\BBB}, L_{\GGG}\right)$-module structure
on the $A_{\BBB}$-module $\Gamma(\theta)$ of sections of $\theta$.
Indeed, let $G$ be a Lie group,
let $\xi\colon P \to \BBB$ be a principal right $G$-bundle
having $\GGG$ as its gauge groupoid
and such that
$\theta$ is the vector bundle
associated with $\xi$ and the left $G$-representation $V$.
Then $L_{\GGG}=\Gamma(\lambda_{\GGG})$ can be taken to be the space
of $G$-equivariant sections of the tangent bundle
$
\tau_p\colon \mathrm TP \longrightarrow P
$
of $P$, the space $\Gamma(\theta)$ of sections of $\theta$
arises as the space of smooth $G$-equivariant maps from $P$ to $V$,
and the $L_{\GGG}$-action on that space of maps is the obvious one
through $G$-equivariant vector fields on $P$.
See also
Theorem 4.1.6 in \cite{mackbtwo} (p.~152),
where this kind of structure is referred to as an {\em infinitesimal
action\/}.

Since the Lie algebroid $\lambda_{\GGG}$ of $\GGG$ is the Atiyah algebroid
of an associated principal $G$-
bundle $\xi\colon P \to \BBB$ on $\BBB$ (such that
$\GGG$ is the gauge groupoid of $\xi$), the corresponding extension
\eqref{2.2.2} of $(\mathbb R,A_{\BBB})$-Lie algebras arising from
the Atiyah sequence \eqref{2.2.1} of $\xi$
is available;
this extension takes the form
\begin{equation}
\mathbf e_{\xi}\colon 
0 \longrightarrow \fra g(\xi) \longrightarrow \mathrm E(\xi)
\longrightarrow \mathrm{Vect}(\BBB) \longrightarrow 0. \label{2.2.2.2}
\end{equation}
Now the cone $C^{\mathbf e_{\xi}}L_{\GGG}$ introduced in Section \ref{ext}
is defined, and
$(A_{\BBB},C^{\mathbf e_{\xi}}L_{\GGG})$ is a differential graded Lie-Rinehart
algebra. We will henceforth write
$C^{\GGG}L_{\GGG}$ rather than $C^{\mathbf e_{\xi}}L_{\GGG}$.
Let
${ }_{\left(A_{\BBB},\GGG, C^{\GGG} L_{\GGG}\right)}\mathrm{Mod}$ 
be the category
of left $(A_{\BBB},\GGG, C^{\GGG} L_{\GGG})$-modules
where the three actions intertwine in a way which we now make precise.

We define a  {\em left\/} $(A_{\BBB},\GGG, C^{\GGG} L_{\GGG})$-{\em module\/}
to be a differential graded left $\GGG$-module $\zeta$
together with a differential graded
left $(A_{\BBB},C^{\GGG}L_{\GGG})$-module structure on the
space $\Gamma(\zeta)$ of sections of the underlying differential graded
vector bundle $\zeta\colon E \to \BBB$;
the two structures are required to be related by (i) and (ii) below:
\newline
\noindent
(i) The 
left $(A_{\BBB},L_{\GGG})$-module structure on 
$\Gamma(\zeta)$ 
induced from the left $\GGG$-structure coincides with
the restriction to $(A_{\BBB},L_{\GGG})$
of the given
left $(A_{\BBB},C^{\GGG}L_{\GGG})$-module structure.
\newline
\noindent
(ii) Relative to the decomposition 
\begin{equation}
C^{\GGG}L_{\GGG}=(s \fra g(\xi))\rtimes L_{\GGG},
\label{decomp}
\end{equation}
the action of the graded $A_{\BBB}$-Lie algebra 
$s \fra g(\xi)$ 
on $\Gamma(\zeta)$ 
is compatible with the left $\GGG$-module structure.

We will now explain the compatibility constraint (ii).
We  refer to a bundle automorphism
of $\xi$ which induces the {\em identity\/} on the base as a
{\em gauge transformation\/}.
Let $\mathcal G(\xi)$ be the group of gauge transformations of $\xi$.
We identify  $\mathcal G(\xi)$ with the group $\mathrm{Map}_G(P,G)$
of $G$-equivariant maps where $G$ acts on itself from the right
by conjugation, the group structure being given by pointwise composition.
The identification is made explicit by the association
\[
\mathrm{Map}_G(P,G)\longrightarrow \mathcal G(\xi),
\ 
\varphi \longmapsto \Phi\colon P \longrightarrow P,
\ \Phi(q) = q \varphi(q),\ q \in P.
\]
Thus $\mathcal G(\xi)$ arises as the space of sections of the associated
fiber bundle
\[
C(\xi)\colon P\times_CG \longrightarrow \BB
\]
where $C\colon G \times G \to G$ is the conjugation action 
$(a,b)\mapsto aba^{-1}$. 
In the language of bisections,  $\mathcal G(\xi)$ 
is the group of bisections of $\GGG$
which take values in the inner automorphism group bundle $C(\xi)$,
cf. Example 1.4.7 on p.~25 of \cite{mackbtwo}.
In particular, a vector field on $P$ whose
flow generates a 1-parameter subgroup of $\mathcal G(\xi)$ is well known
to be identifiable
with a section of the adjoint bundle
$\mathrm{ad}(\xi)$ on $\BBB$.
In this sense, the Lie algebra
$\Gamma(\mathrm{ad}(\xi))$ is the Lie algebra of the group
$\mathcal G(\xi)$.

The group $\mathcal G(\xi)$ acts  from the left on 
any vector bundle $\theta\colon P \times_G V \to \BB$ associated with $\xi$
and the
left $G$-representation $V$ in an obvious manner. 
In the case where the base of $\xi$ is 
a point, 
this action
comes down to the left $G$-structure
on $V$. For general $\xi$, the naturally induced $\mathcal G(\xi)$-action 
\[
C^{\infty}(P,V)^G \times \mathcal G(\xi) \longrightarrow C^{\infty}(P,V)^G
\]
on the space $\Gamma(\theta)\cong C^{\infty}(P,V)^G$ of sections
of $\theta$ is from the right and given by the
assignment to  $(\beta,\alpha) \in C^{\infty}(P,V)^G \times \mathcal G(\xi)$
of the composite $\beta\circ \alpha \in C^{\infty}(P,V)^G$.
We will henceforth argue in terms of the associated 
{\em left\/} $\mathcal G(\xi)$-action 
\begin{equation}
\mathcal G(\xi) \times \Gamma(\theta)  \longrightarrow \Gamma(\theta),
\ (\alpha,\beta)
\longmapsto \alpha(\beta)=\beta\circ \alpha^{-1} .
\label{left}
\end{equation}
In the special case where the base of $\xi$ is 
a point, this $\mathcal G(\xi)$-action comes down to the left $G$-structure
on $V$.

Let $\zeta$ be a differential graded
left $\GGG$-module endowed as well with 
a left $(A_{\BBB},C^{\GGG}L_{\GGG})$-module structure on the
space $\Gamma(\zeta)$ of sections of the underlying differential graded
vector bundle $\zeta\colon E \to \BBB$.
Relative to the decomposition \eqref{decomp},
we write the action with $Y \in L_{\BBB}$ as an operation
\[
\lambda_Y\colon \Gamma(\zeta) \longrightarrow \Gamma(\zeta)
\]
of Lie derivative and 
the action with $v=sY \in s\fra g(\xi)$ as an operation
\[
i_Y\colon \Gamma(\zeta) \longrightarrow \Gamma(\zeta)
\]
of contraction. 
Suppose  that these data satisfy the requirement (i) above. 
Given the section $Y$ of $\mathrm{ad}(\xi)$ (element $Y$ of $\fra g(\xi)$) 
and the gauge transformation $\alpha$ of $\xi$
(element $\alpha$ of $\mathcal G(\xi)$),
the requirement (i) implies that
\[
\alpha \lambda_Y =\alpha  \lambda_Y \alpha^{-1}\alpha 
=\lambda_{\mathrm{Ad}_{\alpha}(Y)}\alpha \colon
\Gamma(\zeta) \longrightarrow \Gamma(\zeta).
\]
Accordingly, the compatibility requirement (ii) takes the form
\[
\alpha i_Y =i_{\mathrm{Ad}_{\alpha}(Y)}\alpha\colon 
\Gamma(\zeta) \longrightarrow \Gamma(\zeta),
\]
where $\alpha$ ranges over gauge transformations and $Y$
over sections of $\mathrm{ad}(\xi)$.

The cone $(A_{\BBB}, C^{\GGG} L_{\GGG})$ on
$(A_{\BBB}, L_{\GGG})$
in the category of differential graded Lie-Rinehart algebras
has an obvious underlying geometric object:
Relative to the principal bundle $\xi\colon P \to \BBB$, consider the trivial
graded vector bundle $P \times s \fra g \to P$ 
whose fibers $\{q\}\times s\fra g$ ($q \in P$) are concentrated in degree 1,
as indicated by the notation $s \fra g$
where $s$ refers to the suspension, and introduce the 
{\em cone\/} on $\lambda_{\GGG}$ as the
differential graded
Lie algebroid which, as a graded vector bundle, is the Whitney sum
\[
((\mathrm T P)\big/G) \oplus P \times_G s \fra g \longrightarrow \BBB
\]
of vector bundles on $\BBB$,
the Lie algebroid structure corresponding exactly to the
cone $(A_{\BBB}, C^{\GGG} L_{\GGG})$
in the category of differential graded Lie-Rinehart algebras.
In particular, the anchor is induced by the composite
\[
((\mathrm T P)\big/G) \oplus P \times_G s \fra g 
\longrightarrow (\mathrm T P)\big/G
\longrightarrow \mathrm T \BBB
\]
of the projection to
$(\mathrm T P)\big/G$ with the anchor of the Lie algebroid $\lambda_{\GGG}$
of $\GGG$.
We refer to the resulting differential graded Lie algebroid
as the {\em cone\/} on $\lambda_{\GGG}$ in the category of differential
graded Lie algebroids.

\section{Equivariant de Rham cohomology over a Lie \\groupoid}
\label{equivgroupoid}
For intelligibility
we begin with recalling the monadic description
of ordinary equivariant de Rham cohomology developed
in our paper \cite{koszultw}.

Let $G$ be an ordinary Lie group, viewed as
a left $G$-manifold via left translation.
Given the chain complex $V$,
the bicomplex $(\mathcal A^*(G,V_*),\delta, d)$ is defined,
where $\delta$ refers to the de Rham complex operator and $d$ to 
the differential induced by the differential of $V$;
let $\mathcal A(G,V)$, the $V$-{\em valued (totalized)
de Rham complex of $G$\/}, be the 
chain complex arising from $(\mathcal A^*(G,V_*),\delta, d)$ by totalization.
The notion of $(G,C\fra g)$-module
and the category $\mathrm{Mod}_{(G,C \fra g)}$
of right  $(G,C\fra g)$-modules have been introduced in 
\cite{koszultw}. 
Endow $\mathcal A(G,V)$
with the right $(G,C\fra g)$-module structure
coming from the derivative of the $G$-action
on itself and the operations of contraction and Lie derivative.
See \cite{koszultw} for details.
Let $\chain$ be the category of real chain complexes.
Consider the {\em pair of categories\/}
\[
\left(\mathcal M,\Cat\right)=
\left(\mathrm{Mod}_{(G,C \fra g)},\chain\right)
\]
and let
\[
\mathcal G_{(G,C \fra g)} \colon \chain \longrightarrow  \mathrm{Mod}_{(G,C\fra g)}
\]
be the functor
which assigns to the chain complex 
$V$ the right $(G,C\fra g)$-module
\[
\mathcal G_{(G,C \fra g)}V= \mathcal A(G,V).
\]
In view of an observation in \cite{koszultw},
the functor $\mathcal G_{(G,C \fra g)}$
is right adjoint to the forgetful functor $\square 
\colon \mathrm{Mod}_{(G,C\fra g)}\to  \chain$
and hence defines a monad  $(\mathcal T,\uuu,\mu)$  over the category 
$\mathrm{Mod}_{(G,C\fra g)}$.
In degree zero, this construction comes down to the monad
spelled out in Example \ref{ex2} above.

Let $\VV$ be  a right $(G,C\fra g)$-module. 
The chain complex $\left|\mathbf T(\VV)\right|$ arising from the
{\em dual standard construction\/}
$\mathbf T(\VV)$ associated with 
the monad $(\mathcal T, \uuu,\mu)$ and 
the right $(G,C\fra g)$-module
$\VV$
is a resolution of $\VV$
in the category of right
$(G,C\fra g)$-modules
that is {\em injective relative to the category of chain complexes\/}.
Given a right  
$(G,C\fra g)$-module $\mathbf W$, by definition,
the {\em differential graded\/}
$\mathrm{Ext}_{((G,C \fra g);\chain)}(\mathbf W,\VV)$ is the homology of
the chain complex
\[
\mathrm{Hom}_{{(G,C\fra g)}}\left(\mathbf W,\left|\mathbf T(\VV)\right|\right).
\]
In particular, relative to the obvious 
trivial $(G,C \fra g)$-module structure
on $\mathbb R$, the  differential graded
$
\mathrm{Ext}_{((G,C \fra g);\chain)}(\mathbb R,\VV)
$ 
is the homology of
the chain complex
\[
\mathrm{Hom}_{(G,C \fra g)}\left(\mathbb R,\left|\mathbf T(\VV)\right|\right)\cong
\left|\mathbf T(\VV)\right|^{(G,C\fra g)}
\]
of $(G,C\fra g)$-invariants in $\left|\mathbf T(\VV)\right|$.

Let $X$ be a left $G$-manifold. The de Rham complex $\mathcal A(X)$
acquires a right $(G,C \fra g)$-module structure
in an obvious manner via the derivative of the
$G$-action and the operations of contraction and Lie derivative, cf.
\cite{koszultw}.
One of the main results of \cite{koszultw} says that
the $G$-equivariant de Rham cohomology of
$X$ is canonically isomorphic to the 
 differential
$\mathrm{Ext}_{((G,C \fra g);\chain)}(\mathbb R,\mathcal A(X))$. 
However, for reasons explained in Remark \ref{cd5},
the functor
$\mathcal G_{(G,C \fra g)}$
does not extend to Lie groupoids
since it does not even extend in degree zero,
whence the
monad  $(\mathcal T,\uuu,\mu)$  over the category 
$\mathrm{Mod}_{(G,C\fra g)}$
cannot extend to a monad defined suitably in terms of groupoids.
The cure is provided by an extension of the monads spelled out
in Examples \ref{ex4} and \ref{ex9}.

Thus let 
$\mathcal U\colon \mathrm{Mod}_{(G,C\fra g)}
\longrightarrow \mathrm{Mod}_{(G,C\fra g)}$
be the functor which assigns to the right $(G,C\fra g)$-module $\VV$ the
right $(G,C\fra g)$-module $\mathcal U(\VV)=\mathcal A(G,\VV)$,
the $(G,C\fra g)$-module structure being
the right {\em diagonal\/} 
$(G,C\fra g)$-module structure given in \cite{koszultw}.
This structure extends the 
action 
\eqref{diag222} coming from left translation in $G$ and
the right  $(G,C\fra g)$-module structure on $\VV$.
Let $\omega$ be the natural transformation given by
the assignment to the right $(G,C\fra g)$-module $\VV$ of
\begin{equation}
\omega=\omega_{\VV}\colon \VV \longrightarrow \mathcal A^0(G, \VV),
\ v \longmapsto \omega_v:G \to \VV,\ \omega_v(x) =v, \ v \in \VV, x
\in G,
\label{2241}
\end{equation}
and let
\begin{equation}
\nu\colon \mathcal U^2 \longrightarrow \mathcal U
\label{nuu}
\end{equation}
be 
the natural transformation
given by the assignment to
the right $(G,C\fra g)$-module $\VV$ of the association
\[
\nu_{\VV} 
\colon \mathcal A(G,\mathcal A(G,\VV))\cong  \mathcal A(G\times G, \VV)
\longrightarrow \mathcal A(G,\VV)
\] 
induced by the diagonal map $G \to G\times G$.
The system $(\mathcal U,\omega, \nu)$
is a monad over the category $\mathrm{Mod}_{(G,C\fra g)}$.
The
{\em dual standard construction\/}
$\mathbf U(\VV)$ associated with 
the monad $(\mathcal U,\omega, \nu)$ and 
the right $(G,C\fra g)$-module
$\VV$ together with $\omega_{\VV}\colon \VV \to \mathbf U(\VV)$
yields the resolution 
\begin{equation}
|\mathbf U(\VV)|
\label{reso2}
\end{equation}
of $\VV$
in the category of right
$(G,C\fra g)$-modules
that is {\em injective relative to the category of chain complexes\/}
and hence defines 
the differential
$\mathrm{Ext}_{((G,C \fra g);\chain)}$ as well.
In particular, given the left $G$-manifold $X$, when we substiotute
for $\VV$ the
de Rham complex $\mathcal A(X)$, we again
obtain the $G$-equivariant de Rham cohomology
$\mathrm H_{G}(X)$
of $X$ 
as $\mathrm{Ext}_{((G,C \fra g);\chain)}(\mathbb R,\mathcal A(X))$. 

Guided by the construction in Example \ref{ex9} above,
we now switch from right $(G,C\fra g)$-modules,
introduced in \cite{koszultw},
to left $(G,C\fra g)$-modules.
Let $\VV$ be a chain complex, endowed with a
left $G$-chain complex structure and with a left
$(C\fra g)$-module structure.
We write the action with $Y \in \fra g$ as a Lie derivative 
$\lambda_Y\colon \VV \to \VV$
and the action with $sY\in s\fra g$ as a contraction $i_Y\colon \VV \to \VV$.
We remind the reader that, as a graded Lie algebra,
$C\fra g = (s\fra g)\rtimes \fra g$; cf. the beginning of Section
\ref{lierine} above.
The $G$-and $(C\fra g)$-module structures are said to combine to
a {\em left\/} $(G,C \fra g)$-module structure when
the left $\fra g$-action is the derivative of the left $G$-action
and when
\[
x(i_Y(v))= (i_{\mathrm{Ad}_xY}(xv)),\ x \in G,\ Y \in \fra g,\ v \in \VV.
\]
We note that, when $\VV$ is a  left  $(G,C \fra g)$-module, plainly
\[
x(\lambda_Y(v))= (\lambda_{\mathrm{Ad}_xY}(xv)),\ x \in G,\ Y \in \fra g,\ v \in \VV.
\]
The crucial example of a left  $(G,C \fra g)$-module
is the de Rham complex $\mathcal A(Q)$ of a smooth manifold
$Q$ acted upon from the right by $G$,
the left $(C\fra g)$-module structure on  $\mathcal A(Q)$
being given by the operations of contraction and Lie derivative.
More generally, given the right $G$-manifold $Q$ and the left
$(G,C \fra g)$-module $\VV$, the 
$\VV$-valued de Rham complex $\mathcal A(Q,\VV)$
acquires a diagonal left $(G,C \fra g)$-module structure.
For the case of right $(G,C\fra g)$-modules
rather than left $(G,C\fra g)$-modules, the diagonal structure
is given in \cite{koszultw}.
With the obvious notion of morphism, 
left $(G,C \fra g)$-modules constitute a category 
${}_{(G,C\fra g)}\mathrm{Mod}$.

We now extend the monad  $(\mathcal U,\omega, \nu)$
over the category $\mathrm{Mod}_{(G,C\fra g)}$
reproduced above, but over the category 
${}_{(G,C\fra g)}\mathrm{Mod}$:
Let $\xi \colon P \to B$ be a principal right $G$-bundle.
Further, let
\[
\mathcal U_{\xi}\colon {}_{(G,C\fra g)}\mathrm{Mod}
\longrightarrow {}_{(G,C\fra g)}\mathrm{Mod}
\]
be the functor which assigns to the left $(G,C\fra g)$-module $\VV$ the
left $(G,C\fra g)$-module 
\[
\mathcal U_{\xi}(\VV)=\mathcal A(P,\VV),
\]
the $(G,C\fra g)$-module structure on $\mathcal A(P,\VV)$ being
the left {\em diagonal\/} 
$(G,C\fra g)$-module structure
coming from the right action of $G$ on $P$ and
the left $(G,C\fra g)$-module structure on $\VV$.
This left diagonal $(G,C\fra g)$-module structure 
on $\mathcal A(P,\VV)$ extends the 
action 
\eqref{diag2222}.
Let $\omega$ be the natural transformation given by
the assignment to the left $(G,C\fra g)$-module $\VV$ of
\begin{equation}
\omega=\omega_{\VV}\colon \VV \longrightarrow \mathcal A^0(P, \VV),
\ v \longmapsto \omega_v:P \to \VV,\ \omega_v(y) =v, \ v \in \VV, y
\in P,
\label{22241}
\end{equation}
and let
\begin{equation}
\nu\colon \mathcal U_{\xi}^2 \longrightarrow \mathcal U_{\xi}
\label{nnuu}
\end{equation}
be 
the natural transformation
given by the assignment to
the left $(G,C\fra g)$-module $\VV$ of the association
\[
\nu_{\VV} 
\colon \mathcal A(P,\mathcal A(P,\VV))\cong  \mathcal A(P\times P, \VV)
\longrightarrow \mathcal A(P,\VV)
\] 
induced by the diagonal map $P \to P\times P$ of $P$.
The system $(\mathcal U_{\xi},\omega, \nu)$
is a monad over the category ${}_{(G,C\fra g)}\mathrm{Mod}$.
Then the
{\em dual standard construction\/}
$\mathbf U_{\xi}(\VV)$ associated with 
the monad $(\mathcal U_{\xi},\omega, \nu)$ and 
the left $(G,C\fra g)$-module
$\VV$ together with $\omega_{\VV}\colon \VV \to \mathbf U_{\xi}(\VV)$
yields the resolution 
\begin{equation}
|\mathbf U_{\xi}(\VV)|
\label{reso22}
\end{equation}
of $\VV$
in the category of left
$(G,C\fra g)$-modules
that is {\em injective relative to the category of chain complexes\/}.
Indeed, $EP/G$ is as well a classifying space for $G$.
More precisely, the injection of $G$ into $P$ as a specific fiber
induces an injection $EG \to EP$ of right simplicial $G$-manifolds;
in fact, $EG$ being a simplicial Lie 
group, the projection $E\to B$ induces
the principal simplicial $EG$-bundle $EP \to EB$.
This simplicial $EG$-bundle, in turn, induces the simplicial fiber bundle
$EP\big/G \to EB$ having fiber $BG=EG/G$. However, since $EB$ is contractible,
the injection $EG/G \to EP/G$ is a homotopy equivalence.
Consequently  $EP/G$ is a classifying space for $G$ as well, whence
\eqref{reso22}
is indeed a resolution of $\VV$
in the category of left
$(G,C\fra g)$-modules
that is injective relative to the category of chain complexes.

In particular, after a choice of base point of $B$ has been made,
the obvious restriction map from 
$|\mathbf U_{\xi}(\VV)|$ to $|\mathbf U(\VV)|$
where  $|\mathbf U(\VV)|$ is a resolution of the kind \eqref{reso2}
but constructed for left $(G,C\fra g)$-modules rather than right ones 
is a comparison map of resolutions.
Hence the resolution 
\eqref{reso22}
defines 
the differential
$\mathrm{Ext}_{((G,C \fra g);\chain)}$ as well.
In particular, given the left $G$-manifold $X$, when we take $\VV$ to be the
de Rham complex $\mathcal A(X)$, 
endowed with the induced left $(G, C\fra g)$-module structure,
the $G$-equivariant de Rham cohomology
$\mathrm H_{G}(X)$
of $X$ 
results as the homology of
\[
\mathrm{Hom}_{(G,C \fra g)}(\mathbb R,|\mathbf U_{\xi}(\mathcal A(X))|).
\] 

Let now $\GGG$ be a locally trivial
Lie groupoid and let $f \colon M \to \BBB$ be a 
left $\GGG$-manifold; thus $f$ is endowed
with a left $\GGG$-action of the kind \eqref{action11}.
Our aim is to give a meaning to
the $\GGG$-{\em equivariant de Rham cohomology\/} $\mathrm H_{\GGG}(f)$ of $f$.
Elaborating somewhat on an observation of Seda's \cite{sedaone}
 quoted before,
we shall concoct a 
functor $\mathcal A(\GGG,\,\cdot \,)$
which assigns to the left 
$(A_{\BBB},\GGG, C^{\GGG} L_{\GGG})$-module $\zeta$---this is a vector bundle
$\zeta\colon E \to \BBB$ with additional structure---a left 
$(A_{\BBB},\GGG, C^{\GGG} L_{\GGG})$-module 
$\mathcal A(\GGG,\zeta)$,
 similar to the 
vector bundle $F(\GGG,\vartheta)$ on $\BBB$ associated
in Section \ref{liegroupoids} above with the left $\GGG$-module $\vartheta$;
however, as a vector bundle on $\BBB$,
over the object $q$ of $\BBB$, the fiber will be
the de Rham complex
$\mathcal A(\GGG^q,E_q)$ of $\GGG^q$ 
having as coefficients the fiber $E_q=\zeta^{-1}(q)$
rather than just the $\vartheta^{-1}(q)$-valued functions
on $\GGG^q$ as in Example \ref{ex9}. 
Substituting for $\zeta$ the appropriate vector bundle
arising from $f$ in a way to be made precise below, we will eventually arrive
at the $\GGG$-equivariant de Rham cohomology $\mathrm H_{\GGG}(f)$ of $f$.

Given the left 
$(A_{\BBB},\GGG, C^{\GGG} L_{\GGG})$-module 
$\zeta \colon E \to \BBB$, 
let
\begin{equation}
\mathcal A(\GGG,\zeta)\colon \mathcal A(\GGG,E)\longrightarrow \BBB
\label{U3}
\end{equation}
be the 
$(A_{\BBB},\GGG, C^{\GGG} L_{\GGG})$-module 
whose underlying vector bundle decomposes, for any object
$q \in \BBB$,
as the induced vector bundle
having 
\[
\GGG^q \times_{\GGG^q_q}\mathcal A(\GGG^q,\zeta^{-1}(q))
\]
as total space; thus, once a choice of  $q \in \BBB$ has been made,
as a vector bundle,
$\mathcal A(\GGG,\zeta)$ amounts to the vector bundle associated with the
principal right $\GGG^q_q$-bundle $\tar \colon\GGG^q \to \BBB$
and the left diagonal $\GGG^q_q$-representation 
on $\mathcal A(\GGG^q,\zeta^{-1}(q))$ induced from the right
$\GGG^q_q$-action on $\GGG^q$ and the 
left $\GGG^q_q$-action on $\zeta^{-1}(q)$.
In Proposition \ref{unravel1} above,
we have spelled out this kind of decomposition
for the smooth functions functor $\mathcal A^0$ with appropriate values
rather than for the entire de Rham functor $\mathcal A$.
Being, as a vector bundle, associated with the
principal bundle $\tar \colon \GGG^q \to \BBB$,
the vector bundle $\mathcal A(\GGG,\zeta)$ acquires
a left $\GGG$-module structure in the standard way, and the operations
of contraction and Lie derivative 
relative to the fundamental vector fields 
coming from the $\GGG^q_q$-action
turn
$\mathcal A(\GGG,\zeta)$ into a left
$(A_{\BBB},\GGG, C^{\GGG} L_{\GGG})$-module.
This module structure extends 
the left diagonal $G$-structure  \eqref{diag2222}
and the left diagonal $(G,C\fra g)$-structure 
on a de Rham complex $\mathcal A(Q,\VV)$
of a smooth right $G$-manifold
mentioned earlier.

Let $\chainB$ be the category of {\em split topological chain complexes
over\/} $\BBB$. Thus an object of $\chainB$ is a differential graded 
topological vector bundle on $\BBB$ such that each differential, that is, 
operator (morphism of topological vector bundles) of the kind 
$d \colon \zeta_k \to \zeta_{k-1}$ splits (i.~e. admits a retraction),
and the morphisms in $\chainB$ are the split morphisms of
split topological chain complexes
over $\BBB$.

The assignment to
the left 
$(A_{\BBB},\GGG, C^{\GGG} L_{\GGG})$-module 
$\zeta$
of the vector bundle
\begin{equation}
\mathcal U(\zeta)=
\mathcal A(\GGG,\zeta)\colon \mathcal A(\GGG,E)\longrightarrow \BBB,
\label{U4}
\end{equation} 
cf. \eqref{U3} above,
endowed with the
left 
$(A_{\BBB},\GGG, C^{\GGG} L_{\GGG})$-module structure 
explained above
is an endofunctor 
\begin{equation}
\mathcal U\colon 
{ }_{\left(A_{\BBB},\GGG, C^{\GGG} L_{\GGG}\right)}\mathrm{Mod}
\longrightarrow 
{ }_{\left(A_{\BBB},\GGG, C^{\GGG} L_{\GGG}\right)}\mathrm{Mod}
\label{eqgroupoidco}
\end{equation}
on ${ }_{\left(A_{\BBB},\GGG, C^{\GGG} L_{\GGG}\right)}\mathrm{Mod}$.
Likewise the construction \eqref{2241}
of the natural transformation $\omega$
and the construction \eqref{nuu}
of the natural transformation $\nu$
extend, and we obtain a monad
$(\mathcal U,\omega, \nu)$
over the category 
${ }_{\left(A_{\BBB},\GGG, C^{\GGG} L_{\GGG}\right)}\mathrm{Mod}$.
Then the
{\em dual standard construction\/}
$\mathbf U(\zeta)$ associated with 
the monad $(\mathcal U,\omega, \nu)$ and 
the left $\left(A_{\BBB},\GGG, C^{\GGG} L_{\GGG}\right)$-module
$\zeta$, together with $\omega_{\zeta}$,
yields the resolution $|\mathbf U(\zeta)|$ of $\zeta$
in the category of left
$\left(A_{\BBB},\GGG, C^{\GGG} L_{\GGG}\right)$-modules
that is {\em injective relative to the category\/} $\chainB$
and hence defines 
the differential
$\mathrm{Ext}_{\left(\left(A_{\BBB},\GGG, C^{\GGG} L_{\GGG}\right);\chainB\right)}$.
Indeed, the chain complex 
\begin{equation}
\left|\mathbf U(\zeta)\right| 
\label{reso3}
\end{equation}
arising from the
{\em dual standard construction\/}
$\mathbf U(\zeta)$ associated with $\zeta$
together with 
\[
\omega_{\zeta}\colon \zeta \longrightarrow \left|\mathbf U(\zeta)\right| 
\]
is an injective resolution of $\zeta$
in the category of left
$(A_{\BBB},\GGG,C^{\GGG} L_\GGG)$-modules
relative to the category $\chainB$.
Given the left
$(A_{\BBB},\GGG,C^{\GGG} L_\GGG)$-module $\eta$, the {\em differential graded\/}
$\mathrm{Ext}_{\left(\left(A_{\BBB},\GGG, C^{\GGG} L_{\GGG}\right);\chainB\right)}(\eta,\zeta)$
is the homology of the chain complex
\[
\mathrm{Hom}_{{(A_{\BBB},\GGG,C^{\GGG} L_\GGG)}}
\left(\eta,\left|\mathbf U(\zeta)\right|\right).
\]

Similarly as in Section \ref{liegroupoids} above,
we now unravel  
the present functor $\mathcal U$.
Let $\xi\colon P \to \BBB$ be a principal right $G$-bundle
whose gauge groupoid is isomorphic to $\GGG$, and
suppose that the vector bundle $\zeta$ 
that underlies the 
$(A_{\BBB},\GGG,C^{\GGG} L_\GGG)$-module under consideration
is associated with $\xi$.
Relative to $\xi$, as a topological vector bundle,
$\zeta$ then amounts to an induced vector bundle of the kind
\[
P\times_G \VV \longrightarrow \BBB,
\] 
for a suitable $(G,C\fra g)$-module $\VV$, 
each constituent of the underlying graded
vector space being suitably topologized.
Then, relative to $\xi$, as a topological vector bundle,
$\mathcal U(\zeta)$
amounts to an induced vector bundle of the kind
\begin{equation}
P\times_G \mathcal A(G,\VV) \longrightarrow \BBB,
\label{induced}
\end{equation} 
each constituent of the $\VV$-valued de Rham complex
$\mathcal A(G,\VV)$ of $G$ being suitably topologized
and turned into a left $(G,C\fra g)$-module via the diagonal action.

Let $\mathcal A_f\colon \mathcal A_f(M) \to \BBB$
be the vector bundle over $\BBB$ having as fiber
$(\mathcal A_f)^{-1}(q)$
over any $q \in \BBB$ 
the ordinary de Rham complex
$(\mathcal A_f)^{-1}(q)=\mathcal A(F_q)$
of the fiber $F_q=f^{-1}(q)$.
As a topological vector bundle,
once a choice of $q \in \BBB$ has been made,
$\mathcal A_f$ plainly amounts to an induced vector bundle of the kind
$\GGG^q\times_{\GGG^q_q} \mathcal A(F_q) \longrightarrow \BBB$.
The left action of $\GGG$ on $f$ induces a left $\GGG$-module structure
on $\mathcal A_f$, and
the operations
of contraction and Lie derivative 
relative to the fundamental vector fields 
coming from the $\GGG^q_q$-action
turn $\mathcal A_f$ 
into a left
$(A_{\BBB},\GGG, C^{\GGG} L_{\GGG})$-module.

We now take the left $\left(A_{\BBB},\GGG, C^{\GGG} L_{\GGG}\right)$-module 
$\zeta$ to be the vector bundle $\mathcal A_f$ and, furthermore, we take
$\eta$ to be
the trivial real line bundle on $\BBB$, 
endowed with the trivial left
$\left(A_{\BBB},\GGG, C^{\GGG} L_{\GGG}\right)$-module structure.
Then
the chain complex
\[
\mathrm{Hom}_{\left(A_{\BBB},\GGG, C^{\GGG} L_{\GGG}\right)}(\eta, |\mathbf U(\zeta)|)
\]
defines the relative 
$\mathrm{Ext}_{\left(\left(A_{\BBB},\GGG,C^{\GGG} L_{\GGG}\right);\chainB\right)}
(\eta,\zeta)$
which, in turn,
we take as the definition of the 
$\GGG$-{\em equivariant de Rham cohomology\/} $\mathrm H_{\GGG}(f)$ of $f$.

\begin{Theorem}
\label{U6}
For any object $q \in \BBB$, restriction induces an isomorphism
of the 
$\GGG$-equivariant de Rham cohomology $\mathrm H_{\GGG}(f)$ of $f$
onto the ordinary $\GGG^q_q$-equivariant de Rham cohomology 
$\mathrm H_{\GGG^q_q}(F_q)$ of the fiber $F_q=f^{-1}(q)$.
\end{Theorem}

\begin{proof}
The argument follows the pattern of the proof of Proposition \ref{prop9}.

Indeed, as noted above, the differential graded vector bundle
$\mathcal U(\zeta)$ comes down to an object of the kind
\eqref{induced}.
Hence, given the trivial 
$\left(A,\GGG,C^{\GGG} L_{\GGG}\right)$-module 
$\eta$ having as underlying vector bundle on $\BBB$ 
the trivial real line bundle,
the chain complex
\[
\mathrm{Hom}_{\left(A_{\BBB},\GGG, C^{\GGG} L_{\GGG}\right)}(\eta, |\mathbf U(\zeta)|)
\]
comes down to the chain complex
\[
\mathrm{Hom}_{(G, C \fra g)}(\mathbb R, |\mathbf U_{\xi}(\VV)|)
\]
which, in turn, as noted above, computes the relative
$\mathrm{Ext}_{((G, C \fra g);\chain)}(\mathbb R, \VV)$.
We may suppose that $f$ is the fiber bundle associated with $\xi$
and the left $G$-manifold $F$.
Substituting $\mathcal A(F)$ for $\VV$ and applying 
one of the main results of \cite{koszultw}, we conclude that
the $\GGG$-equivariant de Rham cohomology 
$\mathrm H_{\GGG}(f)$ of $f$ comes down to the ordinary 
$G$-equivariant de Rham cohomology $\mathrm H_{G}(F)$ of $F$.
This completes the proof of Theorem \ref{U6}.
\end{proof}

\section {Comparison with other notions of Lie groupoid cohomology}
\label{compare}

A possible notion of groupoid cohomology
is given by the ordinary singular cohomology
of the geometric realization of the nerve of the groupoid
under consideration. This kind of cohomology
is explored in e.~g. \cite{crainic}, \cite{weinxu} and elsewhere.
Recall that any stack admits a groupoid presentation and that
the singular 
cohomology of a stack presented by a groupoid is defined to be the
cohomology of the nerve of the groupoid.
Likewise,
given a differentiable stack
presented by the Lie groupoid
$\Lambda$, the nerve $\mathcal N(\Lambda)$ of $\Lambda$
is a simplicial manifold---for $p \geq 0$, we will denote the homogeneous
degree $p$ constituent by $\Lambda_p$---, and the de Rham cohomology of 
the stack
is defined to be the total cohomology
of the cosimplicial de Rham complex
\begin{equation}
\left(\mathcal A^{p,q}(\Lambda),\delta,\varepsilon,\eta 
\right)_{p\geq 0,q \geq 0};
\label{bdss1}
\end{equation}
here
$\mathcal A^{p,q}(\Lambda)=\mathcal A^q(\Lambda_p)$,
the cosimplicial operators $\varepsilon$ and $\eta$
are induced by the ordinary simplicial
operators in the standard manner by dualization
and, for $p,q \geq 0$,
\[
\delta\colon \mathcal A^q(\Lambda_p)\to \mathcal A^{q+1}(\Lambda_p)
\]
is the ordinary de Rham operator. 
See e.~g. \cite{behreone} and the literature there.
We will refer to this cohomology as the {\em stack de Rham cohomology\/}
of $\Lambda$.
As already noted in the introduction,
when the underlying Lie groupoid $\Lambda$ is actually an ordinary Lie group,
viewed as a Lie groupoid with a single object, this cosimplicial
de Rham complex comes down to the construction developed and explored
by Bott, Dupont, Shulman and Stasheff \cite{bottone}, \cite{botshust},
\cite{duponone}, \cite{shulmone}.

Let $\GGG$ be a 
Lie groupoid and $f \colon M \to \BBB$ a 
left $\GGG$-manifold, cf. Section \ref{liegroupoids} above; thus $f$ is endowed
with a left $\GGG$-action
given by a commutative diagram of the kind \eqref{action11}
such that the obvious associativity constraint is satisfied.
Let
$\Lambda=\GGG \ltimes M$ be the associated action groupoid.
As noted above, in the literature, the cohomology
of the associated stack is defined as the singular cohomology 
of the geometric realization $B(\GGG \ltimes M)$ of the nerve
of $\GGG \ltimes M$.
In the  case where the Lie groupoid under discussion is a Lie group,
that geometric realization comes down to the 
ordinary homotopy quotient.
While the de Rham 
functor does not apply to $B(\Omega \ltimes M)$ directly,
the definition in terms of the
total cohomology
of the associated cosimplicial de Rham complex
yields a notion of de Rham cohomology for $\Omega \ltimes M$,
the {\em stack de Rham cohomology of\/}
$\Omega \ltimes M$.
 
More reprecisely, we will write the nerve of $\Omega \ltimes M$
as $\mathcal N(\GGG, f)$. The target map $\tar$ yields
the simplicial manifold 
\begin{equation}
\tar\colon \mathcal N(\GGG, f)\longrightarrow \BBB
\label{simpl}
\end{equation}
over
$\BBB$ where the notation $\tar$ is slightly abused.
The simplicial manifold \eqref{simpl} over $\BBB$ can be seen 
as the standard construction associated with
the corresponding comonad 
that is lurking behind.
The cosimplicial
de Rham complex \eqref{bdss1} now takes the form
\begin{equation}
\left(\mathcal A^{p,q}(\GGG,f),\delta,\varepsilon,\eta 
\right)_{p\geq 0,q \geq 0};
\label{bdss}
\end{equation}
here
$\mathcal A^{p,q}(\GGG,f)=\mathcal A^q(\mathcal N(\GGG,f)_p)$,
the cosimplicial operators $\varepsilon$ and $\eta$
are induced by the ordinary simplicial
operators in the standard manner by dualization,
and
\[
\delta\colon \mathcal A^q(\mathcal N(\GGG,f)_p)
\longrightarrow \mathcal A^{q+1}(\mathcal N(\GGG,f)_p)
\]
is the ordinary de Rham operator. 
The target map $\tar\colon \GGG \to \BBB$
endows $\mathrm{Id}\colon \BBB \to \BBB$
with a left $\GGG$-structure, the {\em trivial\/} structure.
For the special case where  $f=\mathrm{Id}\colon \BBB \to \BBB$,
a de Rham complex of the kind \eqref{bdss} is, indeed, suggested 
in Section 2 of
\cite{botshust} as a possible definition for the cohomology of
the Haefliger groupoid arising in Haefliger's 
classification theory for foliations but, beware,
no such suggestion is made in \cite{botshust}
concerning equivariant cohomology over that groupoid.

To relate the above constructions
with the approach to equivariant cohomology over a groupoid developed
in the present paper,
suppose now that $\GGG$ is locally trivial.
Let $\mathcal A_f\colon \mathcal A_f(M)\to \BBB$ be the differential graded
vector bundle over $\BBB$ having as fiber
$\mathcal A_f^{-1}(q)$ over $q \in \BBB$ the ordinary de Rham complex
$\mathcal A_f^{-1}(q)=\mathcal A(F_q)$ of the fiber $F_q=f^{-1}(q)$.
The construction developed in the present paper starts
from a
differential graded vector bundle of the kind
\begin{equation}
\mathcal A(\GGG,\mathcal A_f)\colon  \mathcal A(\GGG,
\mathcal A_f(M))
\longrightarrow \BBB 
\end{equation}
having as fiber $(\mathcal A(\GGG,\mathcal A_f))^{-1}(q)$
over 
$q \in \BBB$ the ordinary de Rham complex
\[
(\mathcal A(\GGG,\mathcal A_f))^{-1}(q)
=\mathcal A(\GGG^q,\mathcal A(F_q))
\] 
of $\mathcal A(F_q)$-valued forms on the manifold $\GGG^q$
of all morphisms $u$ in $\GGG$ having $q$ as its source, i.e. $\sou(u)=q$.
This is the first step of the construction, and iterating the construction
yields a cosimplicial differential graded
vector bundle $\mathbf U(\mathcal A_f)$ over $\BBB$.
Let $\eta$ be the trivial line bundle on $\BBB$, endowed with the trivial
$\GGG$-structure.
The cohomology is then that given by a complex of the kind
\begin{equation}
\mathrm{HOM}(\eta,|\mathbf U(\mathcal A_f)|)
\label{coho}
\end{equation}
where the notation 
$\mathrm{HOM}$ refers to the requisite structure present in the construction
and not spelled out here and where the notation $|\mathbf U(\mathcal A_f)|$
refers to the chain complex which results from totalization.

The chain complex arising from the cosimplicial de Rham complex
\eqref{bdss}
is substantially different from the chain complex
\eqref{coho}
developed in the present paper,
and the two constructions
are not isomorphic,
perhaps not even related in an obvious manner,
unless the groupoid $\GGG$ is a group.
Suffice it to point out here
that
the operation of applying the functor of taking the space of sections
to the construction in the present paper
yields an object which formally looks somewhat
like the cosimplicial de Rham complex \eqref{bdss}.  
Indeed, a first step towards unveiling
the connections between the two construction
consists, perhaps, in this observation:
given the fiber bundle
$\phi\colon E\to B$, let $C^{\infty}_\phi$ be the vector bundle
on $B$ having as fiber over $b \in B$ the algebra $C^{\infty}(F_b)$ of 
smooth functions on the fiber $F_b$; then, as a $C^{\infty}(B)$-module,
the space of sections of $C^{\infty}_\phi$ 
yields the smooth functions on the total space
$E$.

The cosimplicial de Rham complex
\eqref{bdss} does not arise by the operation
of iterating the construction of a relatively
injective module with respect to the groupoid $\GGG$, though,
unless $\GGG$ is a group
and hence {\em does not yield a kind of injective resolution\/}, and
a suitable {\em monad  leading to that kind of cohomology  is not in sight\/}
since one cannot naively dualize the comonad 
lurking behind the simplicial object \eqref{simpl}
as one does in the 
group case. Indeed, in the group case,
in a sense, the Hochschild-Mostow injective resolution arises
from application of the functor $C^{\infty}$ to the 
corresponding simplicial object, but such a claim cannot be made
over a general Lie groupoid and 
the difficulty was overcome only in \cite{mackeigh}
where the requisite (relatively) injective resolution was exhibited.
Working out the exact features of that kind of construction
and of the 
precise relationship thereof with 
the construction \eqref{coho} is
an interesting endeavor in its own.
However, as noted already in the introduction,
the notions of injective module and injective resolution are fundamental.
They have been developed by the masters to cope with situations where
projective modules are not available.
Without the notion of injective module there would be no sheaf
cohomology, for example. 

One can show in a roundabout way, however, that the
stack de Rham cohomology of $\GGG \ltimes M$
coincides with the 
$\GGG$-equivariant de Rham cohomology $\mathrm H_{\GGG}(f)$ of $f$
developed in the present paper..
Indeed, pick a base point  $q\in \BBB$; 
by Theorem \ref{U6},
restriction induces an isomorphism of
$\mathrm H_{\GGG}(f)$
onto the ordinary $\GGG^q_q$-equivariant de Rham cohomology 
$\mathrm H_{\GGG^q_q}(F_q)$ of the fiber $F_q=f^{-1}(q)$.
Since $\GGG$ is locally trivial,
by Morita equivalence, cf. e.~g. \cite{moermcru},
the action groupoid $\GGG \ltimes M$ is equivalent to
the action groupoid $\GGG^q_q \ltimes F_q$,
restriction induces an isomorphism from the stack de Rham cohomology
of $\GGG \ltimes M$ onto the 
stack de Rham cohomology
of $\GGG^q_q \ltimes F_q$, and the 
stack de Rham cohomology
of $\GGG^q_q \ltimes F_q$ amounts to the ordinary
 $\GGG^q_q$-equivariant cohomology $\mathrm H_{\GGG^q_q}(F_q)$ of  $F_q$.
Putting the various isomorphisms together we
conclude that  the
stack de Rham cohomology of $\GGG \ltimes M$
coincides with the 
$\GGG$-equivariant de Rham cohomology $\mathrm H_{\GGG}(f)$ of $f$.
Thus Theorem \ref{U6} implies that our notion of
$\GGG$-equivariant de Rham cohomology $\mathrm H_{\GGG}(f)$ of $f$
yields a {\em description of the
stack de Rham cohomology of $\GGG \ltimes M$
as a relative derived functor\/}.
This kind of reasoning also reveals that
the real singular cohomology of an action groupoid
of the kind  $\GGG \ltimes M$
(defined as the ordinary singular cohomology of
the geometric realization $B(\GGG \ltimes M)$ of the nerve
of $\GGG \ltimes M$)
coincides with the equivariant de Rham cohomology 
$\mathrm H_{\GGG}(f)$ of $f$.
Furthermore, this reasoning also implies in a roundabout manner
that, suitably interpreted,
the equivariant de Rham cohomology of a locally trivial Lie groupoid
satisfies Morita equivalence.

\section{Outlook}

Various questions remain open including the following ones:

\begin{enumerate}

\item Does the extension of Bott's decomposition lemma given 
in \cite{koszultw} still generalize to the present situation
in the sense that the functor $\mathcal U$ given as
\eqref{eqgroupoidco} induces a decomposition
of the corresponding standard construction into the standard construction
defining the ordinary Lie groupoid cohomology,
spelled out above as Theorem \ref{theo1}, and the standard construction
defining the equivariant Lie algebroid cohomology
that corresponds to the theory developed for an extension
of Lie-Rinehart algebras in Section \ref{ext}? 
If the answer to this question is yes,
what does this extension of the decomposition lemma then signify?

\item Is there a corresponding Chern-Weil construction and, if so, 
what does it signify? More precisely:
Let $G$ be a Lie group and $\fra g$ its Lie algebra.
Filtering the ordinary bar-de Rham bicomplex
of $G$ by the complementary degree we obtain 
a spectral sequence $(\mathrm E_r,d_r)$ explored by
Bott in his seminal paper \cite{bottone}. Bott has in particular shown that
this spectral sequence has
\[
\mathrm E_1=\mathrm H^*_{\mathrm{cont}}(G,\Sigm [(s^2\fra g)^*])
\]
where $\mathrm H^*_{\mathrm{cont}}$ refers to continuous cohomology
and where $\Sigm [(s^2\fra g)^*]$ denotes the symmetric algebra on the dual
of the double suspension $s^2\fra g$ of $\fra g$.
Let now $X$ be a left $G$-manifold
and use the notation $\Sigm^{\mathrm c}$ for the symmetric coalgebra
functor.
In \cite{koszultw} we have extended that spectral sequence to a 
spectral sequence $(\mathrm E_r,d_r)$, referred to in
\cite{koszultw} as the {\em generalized Bott spectral sequence\/},
 having
\[
\mathrm E_1
=\mathrm H^*_{\mathrm{cont}}(G,\mathrm{Hom}(\Sigm^{\mathrm c} [s^2\fra g], \mathcal A(X))).
\]
In particular when $G$ is compact,
$\mathrm H^*_{\mathrm{cont}}(G,\mathrm{Hom}(\Sigm^{\mathrm c} [s^2\fra g], \mathcal A(X)))$ comes down to the $G$-invariants
$\mathrm{Hom}(\Sigm^{\mathrm c} [s^2\fra g], \mathcal A(X))^G$.
But the elements of the latter are precisely homogeneous
$G$-invariant $\mathcal A(X)$-valued polynomial maps on
(the double suspension of) $\fra g$ and these maps, in turn, 
are precisely the 
Chern-Weil maps. Thus the spectral sequence $(\mathrm E_r,d_r)$ explains
the classical Chern-Weil construction.
In \cite{koszultw} we have pushed further: 
One of the results of that paper spells out
the standard complex calculating the $G$-equivariant cohomology of $X$
as a perturbation of the ordinary bar complex operator.
This raises the question to what extent these facts carry over
to general Lie groupoids.
In \cite{extensta} we have worked out 
a preliminary step at the infinitesimal level.

\item How do the notions of duality and modular class
developed in \cite{duality} extend to groupoids?

\item Is there a quasi-version of the theory of the kind developed 
at the infinitesimal level in 
\cite{quasi} and, if so, what does it signify? Is this theory related
with the issue spelled out in the previous item?

\item The functor
$\mathrm{Ext}_{\left(\left(A_{\BBB},\GGG,C^{\GGG} L_{\GGG}\right);\chainB\right)}
$ 
defined in Section \ref{equivgroupoid}
can also be viewed as a functor of the kind
$\mathrm{Ext}_{\left(\left(A_{\BBB},\mathcal G,C^{\GGG} L_{\GGG}\right);\chainB\right)}
$
where $\mathcal G$ is a suitable group of bisections of $\GGG$.
This kind of relative derived functor can presumably be constructed
for general Lie-Rinehart algebras.

\item Do the van Est spectral sequences generalize to Lie groupoids?
We recall briefly the situation for Lie groups.
Let $G$ be a connected Lie group with Lie algebra $\fra g$,
and let $V$ be a $G$-representation.
The first van Est spectral sequence $(E_r,d_r)$ associated with the data has
$E_2$ isomorphic to 
$\mathrm H^*_{\mathrm{diff}}(G,\mathrm H^*_{\mathrm{top}}(G,V))$
and converges to the Lie algebra cohomology $\mathrm H^*(\mathfrak g,V)$.
Here the constituent $\mathrm H^*_{\mathrm{top}}(G,V)$ can be replaced with
the ordinary $V$-valued cohomology $\mathrm H^*_{\mathrm{top}}(K,V)$
of a maximal compact subgroup $K$ of $G$.
The second van Est spectral sequence $(E_r,d_r)$ 
is associated with these data together with a choice
$K$ of compact subgroup $K$ of $G$, 
not necessarily maximal; the spectral sequence has
$E_2$ isomorphic to 
$\mathrm H^*_{\mathrm{diff}}(G,\mathrm H^*_{\mathrm{top}}(G/K,V))$
and converges to the relative Lie algebra cohomology
$\mathrm H^*(\mathfrak g,\mathfrak k; V)$ where $\mathfrak k$ is the Lie
algebra of $K$.
In particular, when $K$ is a maximal compact subgroup,
the homogeneous space $G/K$ is affine and hence contractible,
and the spectral sequence comes down to an isomorphism
$\mathrm H^*_{\mathrm{diff}}(G,V) \cong \mathrm H^*(\mathfrak g,\mathfrak k; V)$,
referred to usually as the van Est isomorphism. This isomorphism, in turn,
then yields a description of an edge map in the first van Est 
spectral sequence.
Which ones of these facts carries over to Lie groupoids and Lie algebroids
is an interesting question. In order to attack these questions
one must first define the requisite notions, for example
relative Lie algebroid cohomology.
A formally appropriate  way to do this might consist in developing
the requisite relative homological algebra, that is,
the appropriate monad.

\item 
What is the most appropriate form of Morita equivalence
relative to the notion of equivariant cohomology
with respect to a locally trivial Lie groupoid?
At the end of the previous section, we identified
the equivariant cohomology developed in the present paper 
with the corresponding stack de Rham cohomology in a roundabout manner.
Is there a more formal way to explain the identification
and to explain the reasons why the two theories are equivalent?

\item Can equivariant de Rham theory be developed 
as a relative derived functor
for Lie groupoids that are not locally trivial,
e.~g. for groupoids of the kind associated with a general foliation
which is not a fiber bundle?
What does the equivariant de Rham cohomology signify in this case?
Presumably the approach would then better be reworked in the 
language of sheaves, the present approach being the special case
where the requisite sheaves are fine.
Is there a description of the stack de Rham cohomology 
over a Lie groupoid that is not locally trivial
as a relative derived functor?
If the answer to both questions is yes, are the two theories the same
or do they differ?

\end{enumerate}
We hope to return to these issues elsewhere.


\begin{thebibliography}{19}
\bibitem{almekump}
R. Almeida and A. Kumpera:
{\pemphas Structure produit dans la cat\'egorie des 
alg\'ebro\"\i des de Lie.}
{\empha An. Acad. Brasil. Cienc.}
\textbf{53} (1981), 247--250

\bibitem{almolone} R. Almeida and P. Molino: {\pemphas Suites d'Atiyah et
feuilletages transversalement complets.} {\empha C. R. Acad. Sci.
Paris I} \textbf {300} (1985), 13--15

\bibitem{atiyaone} M. F. Atiyah:
{\pemphas Complex analytic connections in fibre bundles.}
{\empha Trans. Amer. Math. Soc.} \textbf {85} (1957), 181--207

\bibitem{behreone}  K. Behrend: 
{\pemphas On the de Rham cohomology of differentiable and algebraic stacks.\/} 
{\empha Adv. Math.\/} \textbf {198} (2005),  583--622,
{\tt math/0410255[math.AG]}

\bibitem{bottone}  R. Bott: 
{\pemphas On the Chern-Weil
homomorphism and the continuous cohomology of Lie groups.\/} 
{\empha Adv. Math.\/} \textbf {11} (1973),  289--303

\bibitem {botshust} R. Bott, H. Shulman, and J. Stasheff: {\pemphas
On the de Rham theory of certain classifying spaces.} Adv. Math.
\textbf{20} (1976), 43--56

\bibitem {brchroru}:
U. Bruzzo, L. Chirio, P. Rossi,  and V. N. Rubtsov:
{\pemphas Equivariant cohomology and localization for {L}ie algebroids\/.}
{\empha Funct. Anal. Appl.},
\textbf {43} (2009),
{18--296}, {\tt math/0506392[math.DG]}

\bibitem {canhawei}
A. Cannas de Silva and A. Weinstein:
{\pemphas Lectures on Geometric Models for Noncommutative Algebras.}
Berkeley Mathematics Lecture Notes vol. 10,
Amer. Math. Soc., Providence, R. I. 1999


\bibitem{crainic} M. Crainic: {\pemphas Differentiable and algebroid
cohomology, van Est isomorphisms and characterstic classes},
Comm. Math. Helv. \textbf{78} (2003), 681--721, {\tt math/0008064 [math.DG]}


\bibitem{doldpupp} 
A. Dold und D. Puppe: 
{\pemphas Homologie 
nicht-additiver Funktoren. Anwendungen.\/} 
Annales de l'Institut
Fourier \textbf{11} (1961), 201--313

\bibitem{duponone} J. L. Dupont: {\pemphas Simplicial de Rham
cohomology and characteristic classes of flat bundles.} 
Topology \textbf{15} (1976), 233--245

\bibitem{duskinon}
J.  Duskin:
{\pemphas Simplicial methods and the interpretation of \lq\lq triple\rq\rq\ 
cohomology.\/}
Memoirs Amer. Math. Soc.
\textbf{163} (1975)

\bibitem{godebook}
R. Godement:
 {\bemphas Topologie alg\'ebrique et th\'eorie des faisceaux.\/}
Hermann, Paris (1958)

\bibitem{higgbook}
P. J. Higgins:
{\bemphas Categories and Groupoids.}
Van Nostrand, Princeton, N.J.
(1971)

\bibitem{higgmack}
P. J. Higgins and K. Mackenzie:
{\pemphas Algebraic constructions in the category of Lie algebroids.}
{\empha J. of Algebra}
\textbf{129} (1990), 194--230

\bibitem{higmathr}
P. J. Higgins and K. Mackenzie:
{\pemphas Duality for base-changing morphisms of vector bundles,
modules, Lie algebroids and Poisson structures.}
{\empha Math. Proc. Camb. Phil. Soc.}
\textbf{114} (1993), 471--488

\bibitem{hilstatw}
P. J. Hilton and U. Stammbach:
{\bemphas A Course In Homological Algebra.}
Graduate texts in Mathematics, vol. 4.
Springer, Berlin $\cdot$ Heidelberg $\cdot$ New York
(1971)

\bibitem{hochsone}  G. Hochschild:
{\pemphas Relative homological algebra.}
{\empha Trans. Amer. Math. Soc.} \textbf{82} (1956),
 246--269

\bibitem{hochmost} G. Hochschild and G. D. Mostow:
{\pemphas Cohomology of Lie groups.\/}  Illinois J.  of Math. \textbf {6} 
(1962), 367--401

\bibitem{poiscoho} J. Huebschmann:
{\pemphas Poisson cohomology and quantization.} {\empha J. reine
angew. Math.} \textbf{408} (1990), 57--113

\bibitem{souriau} J. Huebschmann:
{\pemphas On the quantization of Poisson algebras.} In: Symplectic
Geometry and Mathematical Physics, Actes du colloque en l'honneur
de Jean-Marie Souriau, P. Donato, C. Duval, J. Elhadad, G.~M.
Tuynman, eds.; Progress in Mathematics, Vol. 99, Birkh\"auser
Verlag, Boston $\cdot$ Basel $\cdot$ Berlin,  204--233 (1991)

\bibitem{lradq} J. Huebschmann:
{\pemphas Lie-Rinehart algebras, descent, and quantization.}
{\empha Fields Institute Communications}  \textbf{43} (2004),
295--316, {\tt math.SG/0303016}

\bibitem{bv} J. Huebschmann:
{\pemphas Lie-Rinehart algebras, Gerstenhaber algebras, and
Batalin- Vilkovisky algebras.} {\empha Annales de l'Institut
Fourier}  \textbf{48} (1998), 425--440, {\tt math.DG/9704005}

\bibitem{extensta} J. Huebschmann:
{\pemphas Extensions of Lie-Rinehart algebras and the Chern-Weil
construction.} In: Festschrift in honour of Jim Stasheff's 60'th
anniversary, {\empha Cont. Math.}  \textbf{227} (1999), 145--176,
{\tt math.DG/9706002}

\bibitem{duality} J. Huebschmann:
{\pemphas Duality for Lie-Rinehart algebras and the modular
class.} {\empha J. reine angew. Math.}  \textbf{510} (1999), 103--159, 
{\tt math.DG/9702008}

\bibitem{banach} J. Huebschmann:
{\pemphas  Differential Batalin-Vilkovisky algebras arising from
twilled Lie-Rinehart algebras.} {\empha Banach center
publications}  \textbf{51} (2002), 87--102

\bibitem{quasi} J. Huebschmann:
{\pemphas Higher homotopies and Maurer-Cartan algebras:
quasi-Lie-Rinehart, Gerstenhaber-, and Batalin-Vilkovisky
algebras.}  In: The Breadth of Symplectic and Poisson Geometry,
Festschrift in Honor of Alan Weinstein, J. Marsden and T. Ratiu,
eds., Progress in Mathematics, Vol. 232, Birkh\"auser Verlag,
Boston $\cdot$ Basel $\cdot$ Berlin,  237--302 (2004), {\tt
math.DG/0311294}

\bibitem{koszul} J. Huebschmann: {\pemphas Homological perturbations,
equivariant cohomology, and Koszul duality}, 
{\pemphas Documenta math. (to appear)},
{\tt math.AT/0401160}

\bibitem{koszultw} J. Huebschmann: {\pemphas  
Relative homological algebra, homological perturbations,
equivariant de Rham theory, and Koszul duality}, {\tt math.AG/0401161}

\bibitem{mackone} K. C. H. Mackenzie:
{\bemphas Lie groupoids and Lie algebroids in differential geometry},
London Math. Soc. Lecture Note Series, \textbf{124},
Cambridge University Press, Cambridge, England
(1987)

\bibitem{mackbtwo} K. C. H. Mackenzie:
{\bemphas General theory of Lie groupoids and Lie algebroids},
London Math. Soc. Lecture Note Series, \textbf{213},
Cambridge University Press, Cambridge, England
(2005)

\bibitem{mackeigh}
K. A. Mackenzie:
{\pemphas Rigid cohomology of topological groupoids.\/}
J. Austr. Math. Soc. \textbf{26} (1978), 277--301

\bibitem{maclafiv}
S. Mac Lane:
{\pemphas Homologie des anneaux et des modules.}
In: Colloque de topologie alg\'ebrique, Louvain
(1956), 55--80

\bibitem{maclaboo} {S. MacLane}:
{\bemphas Homology.}
{Die Grundlehren der mathematischen Wissenschaften, Band 114},
Springer-Verlag, Berlin (1967)

\bibitem{maclbotw}
{S. Mac Lane}:
{\bemphas Categories for the working mathematician; second edition},
{Graduate Texts in Mathematics}
\textbf{5},
{Springer-Verlag},
{New York} {(1998)}

\bibitem{moermcru}
{I. Moerdijk and J. Mrcun}:
{\bemphas Introduction to foliations and Lie groupoids},
{Graduate Cambridge Studies in Advanced Mathematics}
\textbf{5},
{Cambridge University Press},
{Cambridge} {(2003)}

\bibitem{rinehart} G. Rinehart:
{\pemphas Differential forms for general commutative algebras.}
{\empha  Trans. Amer. Math. Soc.}  \textbf{108} (1963), 195--222

\bibitem{sedaone} A. K. Seda:
{\pemphas An extension theorem for transformation groupoids.}
Proc. Royal Irish Acad. \textbf{75}A (1975), 255--262

\bibitem{gsegatwo} G. B. Segal:
{\pemphas Classifying spaces and spectral sequences.}
Publ. Math. I. H. E. S. \textbf{34} (1968), 105--112

\bibitem{shulmone} H. B. Shulman: {\bemphas Characteristic classes
and foliations.} Ph. D. Thesis,  University of
California (1972)

\bibitem{stashsev} J. D. Stasheff: 
{\pemphas Continuous cohomology of groups and classifying spaces.}
{Bull. Amer. Math. Soc.}
\textbf{84} (1978), {513--530}

\bibitem{vanesthr} W. T. Van Est:
{\pemphas Alg\`ebres de Maurer-Cartan et holonomie.}
{\empha Ann. Fac. Sci. Toulouse Math.}
\textbf{5} (suppl.),  93--134  (1989)

\bibitem{weinxu}
A. Weinstein and P. Xu:
{\pemphas Extensions of symplectic groupoids and quantization.}
{\empha J. reine angew. Math.}
\textbf {417} (1991), 159--189 

\end{thebibliography}
\end{document}